\documentclass[onefignum,onetabnum]{siamart220329}

\usepackage{geometry}
\usepackage{siunitx}
\usepackage{amsmath,amssymb}
\usepackage{amsfonts}
\usepackage{hyperref}
\usepackage{tikz}
\usetikzlibrary{shapes.geometric}
\usepackage{colortbl}
\usepackage{xcolor}
\usepackage{bm}
\usepackage{soul}
\usepackage{siunitx}
\usepackage{lipsum}
\usepackage{amsfonts}
\usepackage{graphicx}
\usepackage{epstopdf}
\usepackage{algorithmic}
\ifpdf
\DeclareGraphicsExtensions{.eps,.pdf,.png,.jpg}
\else
\DeclareGraphicsExtensions{.eps}
\fi

\newsiamremark{remark}{Remark}
\newsiamremark{hypothesis}{Hypothesis}

\usepackage[utf8]{inputenc}
\usepackage[english]{babel}
\usepackage{algorithmic}
\newsiamthm{defn}{Definition}
\newsiamthm{ass}{Assumption}
\newsiamremark{rem}{Remark}

\usepackage{algorithm}
\numberwithin{algorithm}{section}
\numberwithin{figure}{section}
\numberwithin{table}{section}

\newcommand{\rd}{\mathrm{d}}
\newcommand{\re}{\mathrm{e}}
\newcommand{\ri}{\mathrm{i}}

\newcommand{\bfx}{\mathbf{x}}

\newcommand{\N}{\mathbb{N}}
\newcommand{\C}{\mathbb{C}}
\newcommand{\R}{\mathbb{R}}
\newcommand{\T}{\mathbb{T}}

\newcommand{\Z}{\mathbb{Z}}
\newcommand{\bbS}{\mathbb{S}}
\newcommand{\cD}{\mathcal{D}}
\newcommand{\cR}{\mathcal{R}}

\newcommand{\bfb}{\mathbf{b}}
\newcommand{\bfd}{\mathbf{d}}
\newcommand{\bfv}{\mathbf{v}}

\newcommand{\bfP}{\mathbf{P}}
\newcommand{\bfepsilon}{\mathbf{\epsilon}}
\newcommand{\dist}{\mathrm{dist}}

\newcommand{\Dnaive}{\mathfrak{D}}
\DeclareMathOperator*{\argmin}{arg\,min}
\DeclareMathOperator*{\argmax}{arg\,max}

\newcommand{\cond}{\mathrm{cond}}
\newcommand{\errin}{\mathcal{E}_{\mathrm{in}}}
\newcommand{\errout}{\mathcal{E}_{\mathrm{out}}}
\newcommand{\calI}{\mathcal{I}}

\newcommand{\revise}[1]{{#1}}
\newcommand{\pone}{\revise{(C1)}}
\newcommand{\ptwo}{\revise{(C2)}}

\headers{Efficient far-field computation for polygons}{A. Gibbs and S. Langdon}

\title{{An efficient frequency-independent numerical method for computing the far-field pattern induced by polygonal obstacles}}

\author{A. Gibbs\thanks{University College London 
		(\email{andrew.gibbs@ucl.ac.uk}).}
	\and S. Langdon\thanks{Brunel University London
		(\email{stephen.langdon@brunel.ac.uk}).}}

\usepackage{graphicx}

\begin{document}
	\maketitle
	\begin{abstract}
		For problems of time-harmonic scattering by rational polygonal obstacles, embedding formulae express the far-field pattern induced by any incident plane wave in terms of the far-field patterns for a relatively small (frequency-independent) set of canonical incident angles. Although these remarkable formulae are exact in theory, here we demonstrate that: (i) they are highly sensitive to numerical errors in practice, \revise{and} (ii) direct calculation of the coefficients in these formulae may be impossible for particular sets of canonical incident angles, even in exact arithmetic. Only by overcoming these practical issues can embedding formulae provide a highly efficient approach to computing the far-field pattern induced by a large number of incident angles.
		
		Here we \revise{address challenges (i) and (ii), supporting our theory with numerical experiments}. \revise{Challenge} (i) is solved using techniques from computational complex analysis: we reformulate the embedding formula as a complex contour integral and prove that this is much less sensitive to numerical errors. In practice, this contour integral can be efficiently evaluated by residue calculus. \revise{Challenge} (ii) is addressed using techniques from numerical linear algebra: we oversample, considering more canonical incident angles than are necessary{, thus expanding the set of valid \revise{coefficient vectors}. The \revise{coefficient vector} can then be selected using either a least squares approach or column subset selection. }
	\end{abstract}



	\section{Introduction}
	\label{sec:intro}
	
	In problems of {two-dimensional} time-harmonic scattering of {acoustic, electromagnetic or elastic waves,} obtaining a full characterisation of the scattering properties of an obstacle may require a representation of the far-field behaviour induced by a large set, possibly thousands, of incident plane waves~\cite{GaHaHi:12}. In this paper we propose a new method for calculating this representation efficiently across all frequencies for a broad class of polygonal scatterers.
	
	First, we state the scattering problem, which must be solved in order to compute the far-field pattern. We denote each incident plane wave by $u^i(\bfx;\alpha):=\re^{-\ri k(x_1\cos\alpha+x_2\sin\alpha)}$,
	where $\bfx:=(x_1,x_2)\in\R^2$, the wavenumber $k>0$, and the incident angle $\alpha\in[0,2\pi)$. We consider the scattered wave field $u^s(\cdot;\alpha)$ induced by $u^i(\cdot;\alpha)$ and a sound-soft \emph{rational} polygon $\Omega\subset\R^2$, with boundary $\partial\Omega$. Rational polygons have exterior angles that are rational multiples of $\pi$ (see also Definition~\ref{def:RationalPolygons}). The scattered field $u^s(\cdot;\alpha)$ satisfies the Dirichlet Helmholtz problem
	\begin{align}
		(\Delta+k^2)u^s &= 0,\quad\text{in }\R^2\setminus\Omega;\label{Helm}\\
		u^s &= - u^i,\quad\text{on }\partial\Omega;\label{HelmBC}\\
		\frac{\partial u^s(\bfx;\alpha)}{\partial r} - \ri k u^s(\bfx;\alpha) &= o(r^{-1/2}),\quad r:=|\bfx|\to\infty.\label{HelmSRC}
	\end{align}
	
	The \emph{far-field pattern} (also called the \emph{far-field diffraction coefficient} (e.g., \cite{CoKr:13}), \emph{far-field directivity} (e.g., \cite{KrSh:05}), or simply \emph{far-field coefficient} (e.g., \cite{Bi:06}))
	is of practical interest and is central to this paper.
	Intuitively, it describes the distribution of energy of the scattered field, far away from the scatterer.
	We denote the far-field pattern at observation angle $\theta$ (where $\bfx = r(\cos\theta,\sin\theta)$) by $D(\theta,\alpha)$, defined by the asymptotic relationship (see, e.g., \cite[Theorem~2.6]{CoKr:13})
	\begin{equation}\label{FFlabelTest}
		u^s(\bfx;\alpha)= \frac{\re^{\ri(kr +\pi/4)}}{\sqrt{2\pi k r}}\bigg(D(\theta,\alpha)+\mathcal{O}(r^{-1})\bigg),\quad r\rightarrow\infty.
	\end{equation}
	
	We will consider how $D(\theta,\alpha)$ depends on both the observation angle $\theta$ and the incident angle $\alpha$, each considered as elements of the $2\pi$-periodic \revise{set $\bbS:=[0,2\pi)$}, as shown in Figure~\ref{fig:fulld}; we will often write $(\theta,\alpha)\in\T,$ where $\T:=\bbS^2$.
	
	\begin{figure}%
		\centering
		\includegraphics[width=0.6\linewidth]{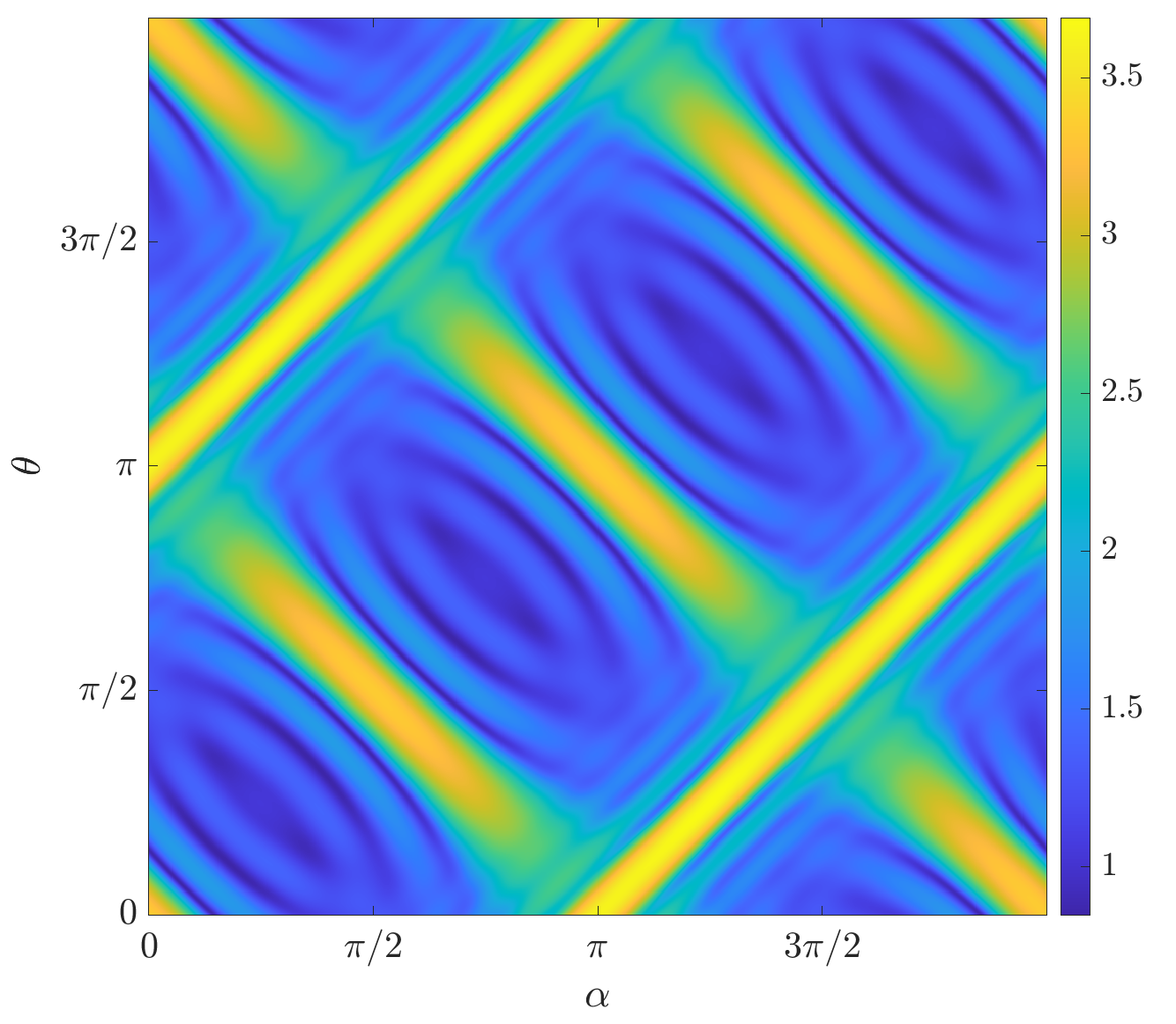}
		\caption{Full far-field characterisation {$\log|D(\theta,\alpha)|$} for $(\theta,\alpha)\in\T$, where $\Omega$ is a \revise{square of diameter 2} and $k=10$. This plot was computed using the approach described in this paper; see \S\ref{sec:numerics} for details.}
		\label{fig:fulld}
	\end{figure}
	{Numerous applications require an understanding of how $D(\theta,\alpha)$ varies over the full range $(\theta,\alpha)\in\T$. For example, in atmospheric physics, when modelling scattering of the sun's radiation by ice crystals in cirrus clouds, to simplify calculations it is common to consider the \emph{orientation average}, where the far-field behaviour is averaged over $\alpha\in\bbS$. Details of the averaging technique and applications can be found in \cite{MiHoTr:00}. %
		Another derived quantity of interest is the \emph{monostatic cross section} $4\pi|D(\alpha,\alpha)|^2$ (sometimes referred to as \emph{backscatter}), where the observer and the source are in the same direction. This \revise{appears} frequently in underwater acoustics and sonar modelling, where it is common for both the signal transmitter and receiver to be positioned on the underside of a boat. A review of relevant applications is given in \cite{GaWe:90}.}
	\revise{This paper considers $\Omega\subset\R^2$, which can provide an approximation for three-dimensional scattering problems on cylindrical obstacles \cite[\S3.4]{CoKr:13}. The potential extension to general three-dimensional obstacles is discussed in \S\ref{sec:future}.}
	
	There are many numerical algorithms for solving \eqref{Helm}--\eqref{HelmSRC} for fixed $\alpha$, \revise{thereby producing} an approximation $D_N(\cdot,\alpha)$ to $D(\cdot,\alpha)$. However, obtaining an approximation for a different incident angle $\alpha'\neq\alpha$ requires the prescription of new boundary data, and hence repetition of some or all of the numerical algorithm; this typically requires a much larger computational cost than varying $\theta$.  This paper is about efficient numerical approximation of $D(\theta,\alpha)$ \revise{over the whole range} $(\theta,\alpha)\in\T$.
	
	A popular approach for approximating $D(\theta,\alpha)$ for $(\theta,\alpha)\in\T$ is the \emph{T-matrix} method (see, e.g., \cite{Ma:06}). Once the $T$-matrix has been constructed, computing the far-field pattern for any given $\alpha\in\bbS$ requires multiplication by a single vector whose entries can be computed with an analytic formula. A drawback though of $T$-matrix methods is that numerically stable construction of the $O(k)$-dimensional $T$-matrix requires, as an input, the numerical solution of $O(k)$ scattering problems \cite{GaHa:09,GaHaHi:12}. Hence, when $k$ is large, $T$-matrix methods may be computationally prohibitive.
	
	\revise{As with} the stable $T$-matrix approach considered in~\cite{GaHa:09}, an input to our method is the numerical solution of {a number of} scattering problems. A key advantage of our method, when compared against $T$-matrix approaches, is that {this number} depends only on the geometry of $\Omega$, and, crucially, is independent of the wavenumber $k$ 
	\revise{in the following sense: if the canonical scattering problems are solved by a numerical method with error $\errin$, then our method will efficiently determine the far-field pattern induced by any incident angle with an error $\errout$, where $\errout/\errin$ is bounded independently of $k$. In particular, this means that if the PDE \eqref{Helm}-\eqref{HelmSRC} is solved by a numerical method for which any prescribed level of accuracy can be achieved with a  number of  degrees of freedom and computational cost that is independent of $k$, as is the case for HNABEM as described in §\ref{sec:hnabem}, then our algorithm will share this key property of the underlying scheme, i.e. that the number of degrees of freedom and computational cost will be independent of $k$.}

	\subsection{The embedding formula for rational polygons}\label{sec:embintro}
	
	Our method is based on an adaptation of the \emph{embedding formula} derived in \cite{Bi:06} (related ideas were first presented in \cite{KrSh:05}). First, we clarify our geometrical constraints.
	
	\begin{defn}[Rational polygons]\label{def:RationalPolygons}
		An angle $\omega\in\bbS$ is \emph{rational} if it can be expressed as $\pi$ multiplied by a rational number. \revise{An $S$-sided polygon $\Omega$ is} rational if its external angles
		$\{\omega_j\}_{j=1}^{S}$ are all rational angles.
		For a rational polygon $\Omega$, we denote by $p$ the smallest {positive} integer such that
		${\pi}/{p}${ divides }$\omega_j${ exactly for }$j=1,\ldots,S,$ %
		whilst $\{q_j\}_{j=1}^{S}$ denotes the set of integers such that $
		{q_j\pi}/{p}=\omega_j$,{ for }$j=1,\ldots,S.$
	\end{defn}
	
	Some examples of rational polygons include: a square with $q_1=\ldots=q_4=3$ and $p=2$; a right-angled isosceles triangle with $q_1=q_2=7$, $q_3=6$ and $p=4$; a screen $\Omega:=[0,1]\times\{0\}$ with $q_1=q_2=2$ and $p=1$. \revise{We note that Definition \ref{def:RationalPolygons} does not require convexity, and this is the case for all of our theoretical results. For simplicity, and due to availability of high-frequency solvers, all of our examples are on convex polygons or screens.}
	
	As in~\cite{Bi:06}, in the remainder of the paper, the formulae have been simplified by assuming that one edge of $\Omega$ is aligned with the horizontal axis. We now state the critical result of~\cite{Bi:06}.
	
	\begin{theorem}\label{th:embedding}
		Suppose that $\Omega$ is a rational polygon and that there exist distinct `canonical incident angles' $\alpha_1,\ldots,\alpha_M$, satisfying Assumption \ref{ass:bexists} {(given below)}, where $M:=\sum_{j=1}^{S}(q_j-1)$ and $q_j$ is as in Definition~\ref{def:RationalPolygons}. Then there exist `embedding coefficients' $b_m(\alpha)$ such that
		\begin{equation}\label{eq:biggs}
			D(\theta,\alpha)=\frac{{\sum_{m=1}^M b_m(\alpha)\Lambda(\theta,\alpha_m){D}(\theta,\alpha_m)}}{\Lambda(\theta,\alpha)},\quad (\theta,\alpha)\in\T,
		\end{equation}
		where
		\begin{equation}
			\Lambda(\theta,\alpha):=\cos(p\theta) - (-1)^p\cos(p\alpha),
			\label{eq:Lambda}
		\end{equation}
		and $p$ is as in Definition \ref{def:RationalPolygons}. If $\Lambda(\theta,\alpha)=0$
		then one or two applications of {L'H\^opital's} rule may be used to express the right-hand side of~(\ref{eq:biggs}) in terms of derivatives with respect to~$\theta$.
	\end{theorem}
	
	The beauty of Theorem~\ref{th:embedding} is that, given far-field patterns $D(\theta,\alpha_m)$ for distinct $\alpha_1,\ldots,\alpha_M$, the  embedding formula \eqref{eq:biggs} provides an exact expression for $D(\theta,\alpha)$, valid for all $(\theta,\alpha)\in\T$. Referring to the example geometries above: $M=8$ for the square, $M=17$ for the right-angled isosceles triangle, and $M=2$ for the screen.
	
	Despite the remarkable implications of Theorem~\ref{th:embedding}, to the best knowledge of the authors, embedding formulae have not found significant use in computational scattering applications. This may be surprising, as the formulae appear to provide a highly efficient and $k$-independent means for computing {$D(\theta,\alpha)$ for all $(\theta,\alpha)\in\T$}. However, although \eqref{eq:biggs} is exact in principle, we will now see that these formulae are incredibly sensitive to numerical errors in the canonical far fields~$D(\theta,\alpha_m$).
	
	We define the numerical analogue of \eqref{eq:biggs}
	\begin{equation}\label{eq:Dnaive}
		\Dnaive_N(\theta,\alpha):=\frac{{\sum_{m=1}^M b_m(\alpha)\Lambda(\theta,\alpha_m){D_N}(\theta,\alpha_m)}}{\Lambda(\theta,\alpha)},
	\end{equation}
	where $D_N(\cdot,\alpha)$ is some numerical approximation to the far-field pattern $D(\cdot,\alpha)$ for $\alpha\in\bbS$, such that $D_N(\cdot,\alpha)\to D(\cdot,\alpha)$ pointwise as $N\to\infty$. We reserve discussion about the approximation of the coefficients $b_m$ until \S\ref{sec:sampling}, for now we assume that they exist and are known exactly.
	
	Algorithmically, \eqref{eq:Dnaive} requires solution of $M$ canonical problems corresponding to incident angles $\alpha_1,\dots,\alpha_M$, after which evaluation for any $(\theta,\alpha)\in\T$ is straightforward. However, for any given $\alpha$, if we consider $\theta\approx\theta_0$, where $\Lambda(\theta_0,\alpha)=0$, it is not hard to see that we immediately run into problems with the representation \eqref{eq:Dnaive}. For the exact theoretical formula \eqref{eq:biggs}, when $\theta=\theta_0$ the theorem states that the value of $D(\theta,\alpha)|_{\theta=\theta_0}$ is determined by the rate at which both the denominator and numerator go to zero as $\theta\to\theta_0$. In the approximate case \eqref{eq:Dnaive}, there are no guarantees that the numerator tends to zero as $\theta\to\theta_0$\revise{; a zero of the numerical approximation is likely close to $\theta_0$, but not at $\theta_0$}. Instead, we expect that the numerator will tend to something small \revise{at $\theta_0$}, approximately zero, but not zero. This \revise{numerical artifact is referred to as a \emph{pole-zero pair}, and} will lead to arbitrarily large errors in $\Dnaive_N(\theta,\alpha)$, because this `small' number is multiplied by unbounded values of $1/\Lambda(\theta,\alpha)$. In this sense~\eqref{eq:Dnaive} is numerically ill-conditioned; this effect can be seen in Figure \ref{fig:illcond}. These qualitative statements are made precise later by Lemma~\ref{lem:HatBound}. \revise{To address these pole-zero pairs, we will reformulate \eqref{eq:biggs} and \eqref{eq:Dnaive} as complex contour integrals. This technique is natural when dealing with removable singularities, see \cite[\S2.1]{AuKrTr:14} and the references [28], [51], and [76] therein.}
	
	\begin{figure}[h]
		\centering
		\includegraphics[width=0.7\linewidth]{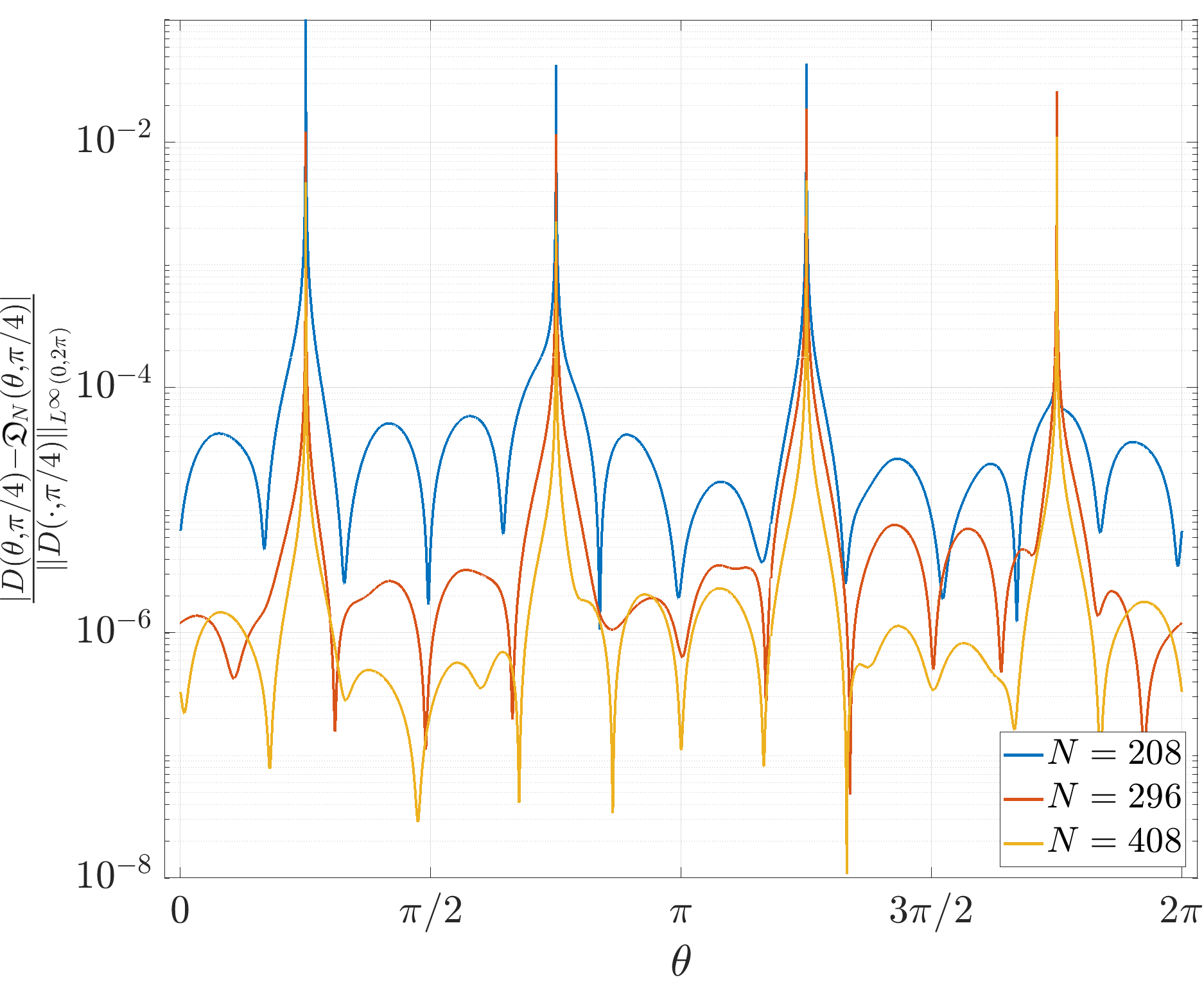}
		\caption{Example of unbounded errors which occur when applying the naive embedding approximation \eqref{eq:Dnaive} directly. Here $\Omega$ is \revise{a square of diameter 2} and $k=10$. {The numerical solver used is described in \S\ref{sec:stdbem}.}}
		\label{fig:illcond}
	\end{figure}
	We summarise this in the first fundamental \revise{challenge} of this paper:
	\begin{itemize}
		\item[\pone] The naive embedding approximation \eqref{eq:Dnaive} is highly sensitive to numerical errors in the canonical far-field patterns.
	\end{itemize}
	
	Henceforth, to simplify the presentation we write
	\begin{equation}\label{eq:hattyboomboom}
		\hat{D}(\theta,\alpha):=\Lambda(\theta,\alpha){D}(\theta,\alpha),\quad\hat{D}_N(\theta,\alpha):=\Lambda(\theta,\alpha){D}_N(\theta,\alpha).
	\end{equation}
	We now focus on calculating the coefficients $\bfb:=[b_1,\ldots,b_M]^T$, which, as we will see, poses a second practical \revise{challenge}.
	Using \eqref{eq:biggs} and the reciprocity principle (see, e.g., \cite[Theorem~3.15]{CoKr:13}), which implies that $\hat{D}(\theta,\alpha) = (-1)^{p+1}\hat{D}(\alpha,\theta)$ for all $ {(\theta,\alpha)\in\T}$, 
	we see that the coefficients $\bfb$ satisfy
	\begin{align}\label{eq:sampling}
		A\bfb=\bfd, \quad\text{where }A:=[\hat{D}(\alpha_m,\alpha_{m'})]_{m,m'=1}^M,\quad\bfd=(-1)^{p+1}[\hat{D}(\alpha,\alpha_{m})]_{m=1}^M.
	\end{align}
	{In {Theorem \ref{th:embedding} and }\cite{Bi:06} the following is assumed:
		\begin{ass}\label{ass:bexists}
			\revise{For any} distinct canonical incident angles $\alpha_1,\ldots,\alpha_M$, the system \eqref{eq:sampling} has a unique solution.
		\end{ass}
	} {Under this assumption, the coefficients $\bfb$ can be derived from readily available quantities, namely the canonical far-field patterns $D(\theta,\alpha_m)$, for $m=1,\ldots,M$. We remark that we have conducted tens of thousands of numerical experiments on varying geometries and wavenumbers, and in each case we have found a set of canonical incident angles $\alpha_1,\ldots,\alpha_M$ satisfying Assumption~\ref{ass:bexists}. However, in many of these experiments we have also found distinct canonical incident angles under which Assumption~\ref{ass:bexists} is not satisfied, {hence the coefficients $\bfb$ cannot be found}}. For example, for the screen problem ($M=2$), if we choose $\alpha_1$ and $\alpha_2$ such that $\cos\alpha_1=-\cos\alpha_2$ then $\Lambda(\alpha_1,\alpha_2)=\Lambda(\alpha_2,\alpha_1)=0$, thus $A$ is the zero matrix and \eqref{eq:sampling} cannot be solved. For a second example, consider the equilateral triangle ($M=12$) with equispaced canonical incident angles $\alpha_m=a + 2(m-1)\pi/M$ for some $a\in\bbS$. Numerical approximation to $A$ suggests that the condition number blows up as $a\to0$; see Figure~\ref{fig:badtricond25p8}.
	
	\begin{figure}
		\centering
		\includegraphics[width=0.5\linewidth]{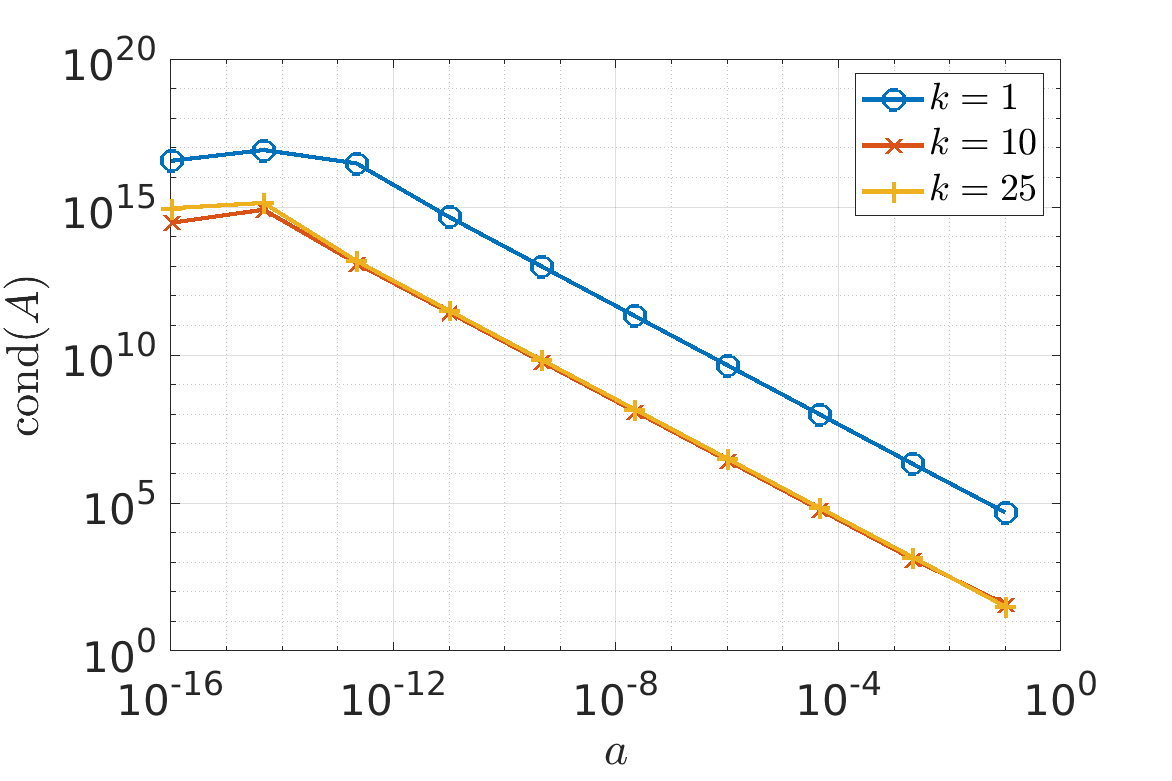}
		\caption{Blow-up of (approximation\revise{s} to) the condition number of $A$ as $a\to0$, for the equilateral triangle with $\alpha_m=a + (m-1)\pi/6$, $m=1,\ldots,12$. The approximation was computed using the method described in \S\ref{sec:stdbem}, with $N=375$.}
		\label{fig:badtricond25p8}
	\end{figure}

	{Even if Assumption \ref{ass:bexists} holds, {$A$ may have a large} condition number. In this case,} the \revise{coefficient} vector $\bfb$ may have large entries, amplifying numerical errors in~\eqref{eq:Dnaive}. These issues are summarised in the second fundamental \revise{challenge} below:
	
	\begin{itemize}
		\item[\ptwo] {In general,} it is unclear how to choose the canonical incident angles so that Assumption~\ref{ass:bexists} holds and hence how to {determine the coefficients $\bfb$}. Even if Assumption~\ref{ass:bexists} holds, it is still unclear how to ensure that the coefficients $\bfb$ have a small norm.
	\end{itemize}
	
	\subsection{Aims and outline of this paper}
	
	In this paper, we address the fundamental \revise{challenge}s {\pone} and {\ptwo}, reformulating the embedding formula~(\ref{eq:Dnaive}) for rational polygons so that it can be of practical use with approximate far-field patterns. %
	
	In \S\ref{sec:fixp1}, we address \pone. We reformulate the embedding formula~(\ref{eq:Dnaive}) as a complex contour integral to reduce the sensitivity to numerical errors. This is quantified by Theorem \ref{thm:semidiscrete_error}. This contour integral may be evaluated by residue calculus, except in a very small set of cases where rounding errors can affect the result. In this case, we interpolate the $k$-dependent part of the integrand by a quadratic polynomial so that the cost of evaluating the integral is independent of $k$. This interpolation is done in such a way that the value of the integral is unchanged. For simplicity, in \S\ref{sec:fixp1} we assume that the embedding coefficients $\bfb$ exist and are exact.
	
	In \S\ref{sec:sampling}, we address \ptwo, and consider the implications of numerical approximation of the embedding coefficients $\bfb$. We cannot solve this \revise{challenge} theoretically, so instead, we oversample, taking more than $M$ canonical far-field patterns{, under the assumption that this gives us a broader space of \revise{coefficient} vectors to choose from. We consider two strategies for selecting a \revise{coefficient} vector from this enhanced space. The first strategy solves the redundant linear system via a regularising truncated singular value decomposition. This approach is partially supported by theory, which informs the choice of the truncation parameter. The second approach chooses a subset of $M$ incident angles which are in some sense optimal, and discards the remaining redundant incident angles. This approach has less theoretical justification, but appears to be more efficient and accurate in practice.}
	
	In \S\ref{sec:numerics}, we present numerical results. These demonstrate the effectiveness of our method via numerical examples, add empirical justification to the approach of \S\ref{sec:sampling} and demonstrate the frequency-independence of our method at high frequencies.
	
	\section{Reformulating the embedding formula}\label{sec:fixp1}
	
	This section addresses \pone. The basic idea of our approach is as follows: It is well-known that the far-field pattern $D(\theta,\alpha)$ is \revise{an entire function} with respect to observation angle $\theta$ (see, e.g., \cite[\S2.2]{CoKr:13}\revise{; note that by reciprocity \cite[Theorem~3.17]{CoKr:13} it is also entire in $\alpha$, but we do not require this here}). \revise{It is then clear from \eqref{eq:Lambda} and \eqref{eq:hattyboomboom} that any finite sum of $\hat{D}(\theta,\alpha)$ is entire in $\theta$.} Hence, we can complexify $\theta$, and, using Cauchy's integral formula, express \eqref{eq:biggs} as a complex contour integral
	
	\begin{equation}\label{eq:CauchyFF}
		D(\theta,\alpha) = \frac{1}{2\pi\ri}\oint_\gamma\frac{\sum_{m=1}^M b_m(\alpha)\hat {D}(z,\alpha_m)}{\Lambda(z,\alpha)(z-\theta)}\rd{z},
	\end{equation}
	where $\gamma$ is any closed contour containing $\theta$, oriented anti-clockwise in $\C$. \revise{This is the only constraint on $\gamma$, recalling that the numerator of the integrand is entire and the singularities of $1/\Lambda(z,\alpha)$ are all removable (Theorem \ref{th:embedding}).} The advantage of \eqref{eq:CauchyFF}, compared to \eqref{eq:biggs}, is that we can choose $\gamma$ in such a way that the magnitude of the denominator in the integrand remains bounded below, away from zero, for $z\in\gamma$. This is in contrast to \eqref{eq:biggs}, where the magnitude of the denominator has no lower bound, and small errors in the numerator lead to arbitrarily large errors overall.
	
	We will show that the approximation
	\begin{equation}\label{eq:CauchyFFN}
		\cD_N(\theta,\alpha; \gamma) = \frac{1}{2\pi\ri}\oint_\gamma\frac{\sum_{m=1}^M b_m(\alpha)\hat{D}_N(z,\alpha_m)}{\Lambda(z,\alpha)(z-\theta)}\rd{z},
	\end{equation}
	is well-conditioned in terms of numerical errors in $\hat{D}_N(z,\alpha_m)$, because we can always choose $\gamma$ to be a safe distance from the poles where $\Lambda(z,\alpha)=0$, and a safe distance from the pole at $\theta$, so that the denominator does not get too large; hence the numerical errors are not significantly amplified. We now calculate the locations of these poles.
	
	\begin{lemma}\label{lem:nearbypoles}
		Given $\alpha\in[0,2\pi)$, the set of poles of $\Lambda{(\cdot,\alpha)}$ is given by
		\[ %
		\Theta_\alpha := \{\theta\in\C:\Lambda(\theta,\alpha)=0\}=\left\{\begin{array}{ll}
			\{\pm\alpha+({2n+1})\pi/{p}: n\in\mathbb{Z}\},&p\text{ odd},\\
			\{\pm\alpha+{2n}\pi/{p}: n\in\mathbb{Z}\},&p\text{ even.}
		\end{array}\right.
		\] 
	\end{lemma}
	\begin{proof}
		Noting that
		\[ \Lambda(\theta,\alpha) = \left\{\begin{array}{ll}
			2\cos\left(\frac{p(\theta+\alpha)}{2}\right)\cos\left(\frac{p(\theta-\alpha)}{2}\right),&p\text{ odd},\\
			-2\sin\left(\frac{p(\theta+\alpha)}{2}\right)\sin\left(\frac{p(\theta-\alpha)}{2}\right),&p\text{ even},
		\end{array}\right.
		\]
		the location of the poles immediately follows.
	\end{proof}
	Note that the elements of $\Theta_\alpha$ are all real.
	
	\begin{defn}[Nearby poles]\label{def:nearbypoles}
		Given $\theta$, we define %
		\[
		\theta_0:=\argmin_{\tilde\theta_0\in\Theta_\alpha}|\theta-\tilde\theta_0|_{2\pi},
		\]
		the closest pole to $\theta$, \revise{where $|\theta-\theta'|_{2\pi}:=\min_{n\in\Z}\{|\theta-\theta'+2n\pi |\}$. Similarly,}
		\[ \theta'_0:=\left\{
		\begin{array}{ll}
			\argmin_{\tilde\theta'_0\in\Theta_\alpha\setminus\{\theta_0\}}|\theta_0-\tilde\theta'_0|_{2\pi},\quad&\theta_0\not\in\{n\pi/p:n\in\Z\}\\
			\theta_0,\quad&\theta_0\in\{n\pi/p:n\in\Z\}
		\end{array}
		\right.
		\]
		\revise{is} the closest pole to $\theta_0$ in the case where $\theta_0'\neq\theta_0$, corresponding to the poles being order one, and $\theta_0'=\theta_0$ corresponding to a pole of order two.
	\end{defn}
	
	Intuitively, it makes sense to choose $\gamma$ such that it encloses $\theta$, remaining as far as possible from the elements of $\Theta_\alpha$ in order to avoid blow-up of the integrand. Doing so may require us to enclose $\theta_0$ and possibly $\theta_0'$ inside $\gamma$, if these poles are close to $\theta$. This is the idea behind the following theorem.
	
	\begin{theorem}\label{thm:semidiscrete_error}
		Suppose the approximations $D_N(z,\alpha_m)$, for $m=1,\ldots,M$, satisfy
		\[
		|D(z,\alpha_m)-D_N(z,\alpha_m)|\leq \epsilon_{m},\quad\text{ for }(z,\alpha_m)\in\Psi\times\bbS
		\]
		where $\Psi:=\{|\Im z|<\ln(3+\pi^2/64)/p\}$ and $\epsilon_m$, $m=1,\ldots,M$, are $N$-dependent constants. Then there exists a closed rectangular complex contour $\gamma$ enclosing $\theta$, $\theta_0$, and $\theta_0'$, such that
		\[
		\left|D(\theta,\alpha)-\cD_N(\theta,\alpha;\gamma)\right|\leq C\|\epsilon\|_2,
		\]
		where $\|\epsilon\|_2:=\sqrt{\sum_{m=1}^M |\epsilon_m|^2}$ and
		\begin{equation}\label{eq:linconst}
			C:=\frac{128\left(5\pi + 4\ln(3+\pi^2/64)\right)\left(6+\pi^2/64\right)}{\pi^4}\|\bfb\|_2.
		\end{equation}
	\end{theorem}
	This theorem is \revise{proved} in \S\ref{sec:proof}, where details of the contour $\gamma$ can also be found.
	
	Theorem \ref{thm:semidiscrete_error} should be interpreted in the following way: if the error in the canonical far-field approximations is bounded, then the error in the contour integral \eqref{eq:CauchyFFN} is bounded. Therefore, the reformulated embedding formula \eqref{eq:CauchyFFN} succeeds where our naive representation \eqref{eq:Dnaive} failed, and we have addressed \pone. It is clear from \eqref{eq:linconst} that the error will also depend on the size and accuracy of the coefficients $\bfb$ - this is addressed in \S\ref{sec:sampling}.
	
	The sceptical reader may point out that: (i) we have analytically extended our approximations ${D}_N(\cdot,\alpha_m)$, a process which is known to be ill-conditioned \cite{trefethen2020quantifying}, \revise{and} (ii) quadrature evaluation of~\eqref{eq:CauchyFFN} may carry a frequency-dependent cost. Fortunately, if we further assume that our numerical approximation $D_N(\theta,\alpha)$ is \revise{analytic for $\theta$ in a complex neighbourhood of $\bbS$} (justified below in Remark \ref{rem:numanal}), we can \revise{choose $\gamma$ to be a closed contour in this neighbourhood containing $\{\theta,\theta_0,\theta_0'\}$ and }evaluate the integral \eqref{eq:CauchyFFN} by residue calculus:
	\begin{align}
		\cD_N(\theta,\alpha,\gamma) =& \frac{\sum_{m=1}^Mb_m(\alpha)\hat{D}_N(\theta,\alpha_m)}{\Lambda(\theta,\alpha)} \nonumber\\
		&-\begin{cases}\displaystyle \sum_{\chi \in\{\theta_0,\theta_0'\}}\frac{ \sum_{m=1}^M b_m(\alpha)\hat{D}_N(\chi,\alpha_m)}{p(\chi-\theta)\sin(p\chi)},&\revise{\theta_0\neq\theta_0',}\\
			\displaystyle\revise{2\frac{\sum_{m=1}^M b_m(\alpha)\left[(\theta_0-\theta)\frac{\displaystyle\partial\hat{D}_N(z,\alpha)}{\displaystyle\partial z}|_{z=\theta_0}-\hat{D}_N(\theta_0,\alpha_m)\right]}{p^2(\theta_0-\theta)^2\cos(p\theta_0)}},&\revise{\theta_0=\theta_0',}
		\end{cases}\quad\theta\not\in\Theta_\alpha.\label{eq:res}
	\end{align}
	\revise{The second case on the right-hand side corresponds to the double pole.}
	As in Theorem~\ref{th:embedding}, at the points $\theta\in\Theta_\alpha$ {we can compute} $\cD_N$ using L'H\^{o}pital's rule.
	
	By representing the integral as \eqref{eq:res}, we address the two \revise{concerns} above: (i) because all poles are in $\Theta_\alpha\cup\{\theta\}$, no analytic continuation is required, \revise{and} (ii) we have evaluated the integral without quadrature.
	
	It is instructive to notice that the first term on the right-hand side of \eqref{eq:res} is precisely \eqref{eq:Dnaive}. Therefore, the second term on the right-hand side of \eqref{eq:res} may be interpreted as a correction to \eqref{eq:Dnaive}. Conversely, in the exact case (where $D_N$ is replaced by $D$), the formula \eqref{eq:res} is exact, because the residues are zero. This is because the points at $\Theta_\alpha$ are removable singularities in theory, manifesting as \revise{pole-zero pairs} in practice.
	
	\begin{rem}[Analyticity of numerical solutions]\label{rem:numanal}
		The assumption required by \eqref{eq:res}, that $D_N(\theta,\alpha)$ is \revise{analytic for $\theta$ in a neighbourhood of $\bbS$}, is entirely reasonable. In finite difference / element / volume and boundary element methods for solving \eqref{Helm}-\eqref{HelmSRC}, the far-field pattern \eqref{FFlabelTest} is approximated by integrating some (typically piecewise-analytic) data against an \revise{entire} kernel (see e.g. \cite{Mo:95} and \cite[\S3.5]{CoKr:13}). Therefore, the resulting approximation is \revise{entire}, and the estimate of Theorem \ref{thm:semidiscrete_error} also applies to the formula~\eqref{eq:res}.
	\end{rem}

	\subsection{Evaluation of residues in finite precision arithmetic}\label{sec:FPA}
	
	Until now, we have not discussed the implications of rounding errors, implicitly assuming that all calculations are done in exact arithmetic. When two poles in \eqref{eq:res} coalesce, the corresponding {residues} will grow, and there will be a large amount of numerical cancellation. In finite precision arithmetic, small rounding errors will be amplified; this is commonly known as \emph{catastrophic cancellation}. We first remark that the region where this occurs is much smaller than the region where \eqref{eq:Dnaive} breaks down (see the discussion around $h$ and $H$ in \S\ref{sec:mainalg}), so even if nothing is done to address this issue, \eqref{eq:res} still offers a significant improvement over~\eqref{eq:Dnaive} in practice. Secondly, we remark that variable precision arithmetic (VPA) may be used to address catastrophic cancellation, whereas VPA would not fix the breakdown of \eqref{eq:Dnaive}. Here, we present a fix for this issue, which may be used without VPA.
	
	Suppose we were to evaluate the integral representation \eqref{eq:CauchyFFN} by numerical quadrature.  This offers the advantage that we could choose $\gamma$ such that the denominator remains bounded below, thus the samples at the quadrature nodes remain bounded, avoiding catastrophic cancellation between the quadrature samples (see \cite{IoPaPe:91} for a summary of relevant numerical integration techniques). However, we recall that a potential disadvantage of the quadrature approach (which motivated \eqref{eq:res}) is that the numerator is an oscillatory function, and as $k\to\infty$ this may grow exponentially and/or oscillate increasingly rapidly along certain segments of $\gamma$. Here lies the dilemma: the contour integral approach can avoid catastrophic cancellation, and the residue approach avoids an $O(k)$ factor increase in computational cost due to quadrature. Can we avoid both?
	
	Surprisingly, thanks to an idea we believe to be new, the answer is \emph{yes}. The idea rests on the following key observation. For large $k$ the \emph{integrand} of \eqref{eq:CauchyFFN} is highly oscillatory (on the real line), but the integral still only depends on the integrand's values at three points: $\theta,\theta_0$ and $\theta_0'$. Therefore, we can interpolate the $k$-dependent oscillatory numerator by a quadratic polynomial $\rho_2$, constructed such that 
	$
	\rho_2(\xi)=\sum_{m=1}^M b_m(\alpha)\hat{D}_N(\xi,\alpha_m)$ for $\xi\in\{\theta,\theta_0,\theta_0'\}.
	$
	In the case where $\theta=\theta_0$ or $\theta_0=\theta_0'$ we add the additional constraint that 
	$
	\rho_2'(\theta_0)=\sum_{m=1}^Mb_m(\alpha)\frac{\partial\hat{D}_N}{\partial \theta}(\theta,\alpha_m)|_{\theta=\theta_0},
	$
	since residues at double poles depend on the derivative of the numerator. Obtaining the (approximate) derivative of the far-field pattern is easy in practice, for reasons similar to those given in Remark~\ref{rem:numanal}.
	
	Barycentric interpolation may be used for efficiency (see e.g. \cite{Hi04,SaVi:35}). However, it is not essential for accurate results, because $\gamma$ can be chosen so that $\rho_2$ is only evaluated at a bounded distance from the interpolation points. It follows that
	\begin{equation}\label{eq:Cauchyrho3}
		\cD_N(\theta,\alpha; \gamma) = \frac{1}{2\pi\ri}\oint_\gamma\frac{\rho_2(z)}{\Lambda(z,\alpha)(z-\theta)}\rd{z},
	\end{equation}
	since the residues of the integral depend only on the value of the integrand (and possibly its derivative) at the poles $\xi=\theta,\theta_0,\theta_0'$. At these poles, the integrands and, where appropriate, the derivatives of the integrand in \eqref{eq:CauchyFFN} and \eqref{eq:Cauchyrho3} are equal, and thus by the residue theorem \eqref{eq:CauchyFFN} and \eqref{eq:Cauchyrho3} have the same value. However, \eqref{eq:Cauchyrho3} is independent of $k$, so accurate evaluation by quadrature requires an $O(1)$ cost, instead of $O(k)$, as $k\to\infty$.
	
	\subsection{Proof of Theorem \ref{thm:semidiscrete_error}}\label{sec:proof}
	
	Before we can prove Theorem~\ref{thm:semidiscrete_error}, we require some preliminary results.  First, we note that for certain values of $\alpha$, pairs of poles in $\Theta_\alpha$ coalesce, forming a pole of order two.
	
	\begin{defn}[Coalescence points]\label{def:coaleset}
		
		We define the set of `coalescence points' by:
		\begin{equation}
			\Theta_*:=\left\{\theta\in\bbS:\frac{\partial{\Lambda}}{\partial\theta}(\theta,\alpha)=0\right\} = \{\theta\in\bbS:\theta=n\pi/p,\; n\in \mathbb Z\}.\label{LHoppyLopps2}
		\end{equation}
		With $\theta_0$ and $\theta_0'$ defined as in Definition~\ref{def:nearbypoles}, we define the `nearest coalescence point' as
		\[
		\theta_*=\argmin_{\tilde\theta_*\in\Theta_*}|\theta_0-\tilde\theta_*|_{2\pi}.
		\]
	\end{defn}
	
	When $\theta_0$ and $\theta_0'$ are close to $\theta_*$, the singularity will be stronger. Therefore these points play an important role when quantifying the breakdown of the naive embedding formula~\eqref{eq:Dnaive}. The following result provides a lower bound on $\Lambda(\theta,\alpha)$, and will be useful when choosing the contour $\gamma$ to avoid amplification of errors in the integral representation~\eqref{eq:CauchyFFN}.
	
	\begin{lemma}\label{lem:HatBound}
		For $\Lambda$ as in \eqref{eq:Lambda}, $\theta_0$ as in Definition \ref{def:nearbypoles} and $\theta_*$ as in Definition \ref{def:coaleset},
		\begin{equation}\label{eq:HatBound}
			\frac{p^2}{8}|\theta-\theta_0||\theta-\theta_*|\leq|\Lambda(\theta,\alpha)|,\quad(\theta,\alpha)\in\T.
		\end{equation}
	\end{lemma}
	\begin{proof}	
		Firstly, by Definition \ref{def:nearbypoles}, we have $
		\Lambda(\theta,\alpha)=
		{\Lambda(\theta,\alpha)-\Lambda(\theta_0,\alpha)}=
		\cos(p\theta)-\cos(p\theta_0)$ 
		and from standard trigonometric identities it follows that
		\begin{equation}\label{Tbound1}
			\left|{{\Lambda}{(\theta,\alpha)}}\right|
			=
			{2\left|\sin\left({p}(\theta-\theta_0)/{2}\right)\right|\cdot\left|\sin\left({p}(\theta_0+\theta)/{2}\right)\right|}.
		\end{equation}
		We first focus on the lower bound of \eqref{eq:HatBound}, for which we will require the inequality
		\begin{equation}\label{eq:dodgyInequality}
			|\sin(p(\theta_*+\theta)/2)|\leq|\sin(p(\theta_0+\theta)/2)|.
		\end{equation}
		
		To see why \eqref{eq:dodgyInequality} holds, we note that the points  $\theta_0\in\Theta_\alpha$ are distributed symmetrically about the points $\theta_*\in\Theta_*$, and we have specified the condition that $\theta_0$ must be the element of $\Theta_\alpha$ closest to $\theta$; hence it follows that $\theta$ and $\theta_0$ both lie on the same side of $\theta_*$ (in a local sense). Hence, it is clear that $p(\theta+\theta_0)/2$ is farther from $p\theta_*$ than $p(\theta+\theta_*)/2$, and this distance is no more than $\pi/(2p)$ by~\eqref{LHoppyLopps2}.
		
		Combining \eqref{Tbound1} with \eqref{eq:dodgyInequality} yields
		\begin{align*}
			|{\Lambda(\theta,\alpha)}|
			&\geq
			{2\left|\sin\left({p}(\theta-\theta_0)/2\right)\right|\cdot\left|\sin\left({p}(\theta_*+\theta)/2\right)\right|}\\
			&={2\left|\sin\left({p}(\theta-\theta_0)/{2}\right)\right|\cdot\left|\sin\left({p}(2\theta_*+(\theta-\theta_*))/{2}\right)\right|}\\
			&={2\left|\sin\left({p}(\theta-\theta_0)/{2}\right)\right|\cdot\left|\sin(p\theta_*)\cos(p(\theta-\theta_*)/2)+\sin(p(\theta-\theta_*)/2)\cos(p\theta_*)\right|}.
		\end{align*}
		By the definition \eqref{LHoppyLopps2} of $\Theta_*$, we have that $\sin(p\theta_*)=0$ and $|\cos(p\theta_*)|=1$, so
		\begin{equation}\label{sinesBeforeBound}
			\left|{\Lambda(\theta,\alpha)}\right|
			\geq{2\left|\sin\left({p}(\theta-\theta_0)/{2}\right)\right|\cdot\left|\sin(p(\theta-\theta_*)/2)\right|}.
		\end{equation}
		From Definition \ref{def:nearbypoles} and~\eqref{LHoppyLopps2}, it follows that the farthest $\theta$ can be from the nearest $\theta_0$ or $\theta_*$ is $\pi/(2p)$. Hence the argument of both sines of \eqref{sinesBeforeBound} is at most $\pi/4$, so we may use the identity $|\sin(x)|\geq|x/2|$ for $0\leq|x|\leq\pi/4$ to obtain the lower bound on $|\Lambda(\theta,\alpha)|$ as claimed.
	\end{proof}
	
	The above result provides an explanation for the blow-up observed in, for example, Figure~\ref{fig:illcond}. With this in mind, we aim to construct the contour $\gamma$ in \eqref{eq:CauchyFFN} so that the denominator in the integrand never gets too {close to zero}. To do this, we will also need a lower bound on $\Lambda$ in the complex plane.
	\begin{lemma}\label{lem:im_change}
		For $\Lambda$ as in \eqref{eq:Lambda} and $c\in\R$,
		\[
		|\Lambda(\theta,\alpha)|\leq|\Lambda(\theta+\ri c,\alpha)|,\quad(\theta,\alpha)\in\T.
		\]
	\end{lemma}
	\begin{proof}
		Define $J(\theta,c):=|\Lambda(\theta+\ri c,\alpha)|$. We are interested in how $\Lambda$ changes as we move from the real line in the positive imaginary direction. Therefore, the following quantity will be useful:
		\[
		\frac{\partial J}{\partial c}(\theta,c) = \frac{p\left((-1)^{p+1}\cos(p\alpha)\cos(p\theta) + \cosh(pc)\right)\sinh(pc)}{\sqrt{\left(\cos({p}\alpha)\cosh(pc)+(-1)^{p+1}\cos(p\alpha)\right)^2 + \left(\sin(p\theta)\sinh(pc)\right)^2}},
		\]
		which can be obtained by using the representation
		\[
		\Lambda(\theta+\ri c,\alpha) = \cos(p\theta)\cosh(pc) - \ri \sin(p\theta)\sinh(pc) - (-1)^p\cos(p\alpha).
		\]
		The denominator is clearly positive, because $\theta,\alpha$ and $c$ are all real. We focus on the sign of the numerator. We have $\cos(p\alpha)\cos(p\theta) \geq-1$ and, for all $c\neq0$, $\cosh(pc)>1$, hence
		\[\left((-1)^{p+1}\cos(p\alpha)\cos(p\theta) + \cosh(pc)\right)\sinh(pc)
		\left\{\begin{array}{ll} >0 & \text{when } c>0 \\ = 0 & \text{when } c=0 \\ <0 & \text{when } c<0. \end{array}\right. \]
		Hence $|\Lambda(\theta+\ri c,\alpha)|$ is minimised when $c=0$, proving the assertion.
	\end{proof}
	
	The above result tells us that $|\Lambda(z,\alpha)|$ increases as we move vertically into the complex plane. This a helpful result for evaluating the contour integral~\eqref{eq:CauchyFFN}, because numerical errors are amplified when $|\Lambda|$ is small. Considering that $\Lambda$ is in the numerator and the denominator of~\eqref{eq:CauchyFFN}, we also require the following bound.
	\begin{lemma}\label{lem:lamupperbd}
		For $\Lambda$ as in \eqref{eq:Lambda},
		\[
		\frac{\re^{p|\Im z|}-3}{2} \leq \left|\Lambda(z,\alpha)\right| \leq\frac{\re^{p|\Im z|} + 3}{2},\quad\text{for }(z,\alpha)\in\C\times\bbS.
		\]
	\end{lemma}
	\begin{proof}
		By the triangle inequality and the definition \eqref{eq:Lambda},
		\begin{align*}
			|\Lambda(z,\alpha)|&\leq|\cos(pz)|+1{=}\left|\frac{\re^{\ri p z}+\re^{-\ri p z}}{2}\right| + 1.
		\end{align*}
		Similarly, by the negative triangle inequality,
		\begin{align*}
			|\Lambda(z,\alpha)|&\geq|\cos(pz)|-1{=}\left|\frac{\re^{\ri p z}+\re^{-\ri p z}}{2}\right| - 1.
		\end{align*}
		The assertion follows by considering the cases $\Im z$ positive and negative.
	\end{proof}
	
	We are now ready to prove the main result of this section.
	\begin{proof}[Proof of Theorem \ref{thm:semidiscrete_error}]

		\begin{figure}
			\centering	
			\begin{tikzpicture}[scale=0.9]
				\def\L{-4.5}
				\def\R{6}
				\draw[dashed] (\L,-1.5) -- (\R-.5,-1.5) -- (\R-.5,1.5) -- (\L,1.5) -- (\L,-1.5);
				\node[above left] at (\L,1.5) {$\gamma$};
				\filldraw (2,-0.05) rectangle ++(4pt,4pt);\node[above ] at (2,0) {$\theta_*$};
				\filldraw (-3,-0.05) rectangle ++(4pt,4pt);\node[above ] at (-3,0) {$\theta_*-\pi/p$};
				\filldraw (\R+1,-0.05) rectangle ++(4pt,4pt);\node[below] at (\R+1,0) {$\theta_*+\pi/p$};
				\draw (0,0) circle[radius=2pt, fill=white];\node[above] at (0,0) {$\theta_0$};
				\draw (4,0) circle[radius=2pt, fill=white];\node[above] at (4,0) {$\theta_0'$};
				\draw (-6,0) circle[radius=2pt];\node[above] at (-6,0) {$\theta_0'-2\pi/p$};
				{		\color{blue}\filldraw (-1,0) circle[radius=1.5pt];\node[above] at (-1,0) {$\theta$};\node[below] at (-1,0) {(case one)};
					\filldraw (1,0) circle[radius=1.5pt];\node[above] at (1,0) {$\theta$};\node[below] at (1,0) {(case two)};}
				\node[above left] at (-6,2) {$\C$};
				\draw [->,line width=.9pt](-5-1.5,0) -- (\R+3,0);\node[right] at (\R+2,0.35) {$\Re(z)$};
				\draw [<->,line width=.5pt](\R-0.45,0.1) -- (\R-0.45,1.5);\node[right] at (\R-0.45,0.75,0) {$=\frac{\ln(3+\pi^2/64)}{p}$};
				\draw [<->,line width=.5pt](-3,0.6) -- (0,0.6);\node[above] at(-1.5,0.6) {$\geq\frac{\pi}{2p}$};
				\draw [<->,line width=.5pt](-6,-0.6) -- (\L,-0.6);\node[below] at(-5.25,-0.6) {$\geq\frac{\pi}{4p}$};
				\draw [<->,line width=.5pt](\L,-0.6) -- (-3,-0.6);\node[below left] at(-3,-0.6) {$\geq\frac{\pi}{4p}$};
				\draw [<->,line width=.5pt](-3,-0.6) -- (2,-0.6);\node[below] at(-1,-0.6) {$=\frac{\pi}{p}$};
				
				\draw (\L,0) circle[radius=2pt, fill=black];\node[right] at(\L,-0.2) {$z_-$};
				\draw (\R-0.5,0) circle[radius=2pt, fill=black];\node[left] at(\R-0.5,-0.2) {$z_+$};
			\end{tikzpicture}
			\caption{Rectangular contour, covering both cases one and two of the proof of Theorem \ref{thm:semidiscrete_error}. The arrows are labelled with values (or bounds) of the length of the regions to which the arrows are parallel.}\label{fig:proofcases}
		\end{figure}
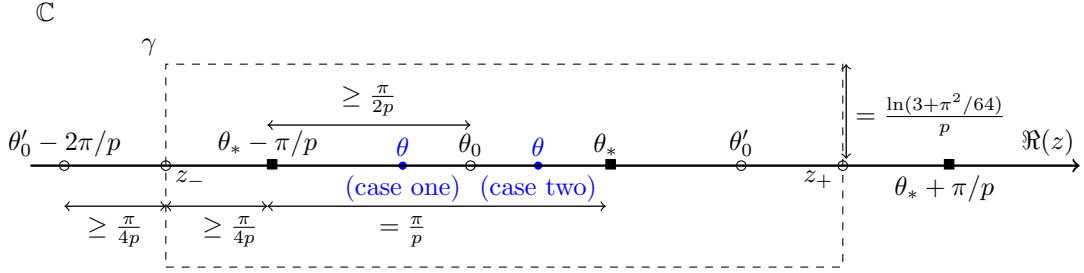
		\newcommand{\cM}{\mathcal{M}}
		We aim to continuously deform onto a rectangular complex contour $\gamma$ (as in Figure \ref{fig:proofcases}), such that we can bound below the denominator of~\eqref{eq:CauchyFF}, and~\eqref{eq:CauchyFFN}, uniformly on $\gamma$.  We denote this bound by $\cM(\gamma):=\min_{z\in\gamma}\{\Lambda(z,\alpha)(z-\theta)\}$.
		Using this bound, which will be derived shortly, together with Lemma~\ref{lem:lamupperbd}, taking the absolute value of the integrand and applying the \revise{Cauchy-Schwarz} inequality, we can obtain the following estimate:
		
		{\begin{equation}
				\hspace{-2pt}\left|\frac{1}{2\pi\ri}\oint_\gamma\frac{\sum_{m=1}^M b_m(\alpha)\Lambda(z,\alpha_m)[D(z,\alpha_m)-D_N(z,\alpha_m)]}{\Lambda(z,\alpha)(z-\theta)}\rd{z}\right|\leq
				\frac{|\gamma|(3+\re^{p\max|\Im\gamma|})}{4\pi\cM(\gamma)}
				\sqrt{\sum_{m=1}^M |b_m(\alpha)|^2\epsilon_m^2}.\label{eq:big_est}
		\end{equation}}
		It is sufficient to consider two cases, as depicted in Figure \ref{fig:proofcases}, the rest follow by symmetry.
		
		Case one is where $\theta\leq\theta_0\leq\theta_*$. By definition, $\theta_*$ is closer to $\theta_0$ than $\theta_*-\pi/p$ is to $\theta_0$, hence $|(\theta_*-\pi/p)-\theta_0|\geq\pi/(2p)$. As the elements of $\Theta_\alpha$ are symmetric about the elements of $\Theta_*$, we can then deduce that $|(\theta_*-\pi/p)-(\theta_0'-2\pi/p)|\geq\pi/(2p)$. If we choose $\gamma$ to bisect the real line at $z_-$, the point halfway between $(\theta_*-\pi/p)$ and $(\theta_0'-2\pi/p)$, then we have that $\mathrm{dist}(\Theta_\alpha,z_-)\geq \pi/(4p)$ and $\mathrm{dist}(\Theta_*,z_-)\geq \pi/(4p)$. Since $\theta\geq(\theta_*-\pi/p)$, we also have $\mathrm{dist}(z_-,\theta)\geq \pi/(4p)$. A similar construction holds in the region to the right, choosing $\gamma$ to intersect $\R$ at the point $z_+$, which is halfway between $\theta_0'$ and $\theta_*+\pi/p$. This is indicated by Figure \ref{fig:proofcases}, and the bound $\mathrm{dist}(z_+,\theta)\geq \pi/(4p)$ follows similar arguments.
		
		Case two is where $\theta_0\leq\theta\leq\theta_*$. If we use the same contour as \revise{in} case one, we can make the stronger statement $\mathrm{dist}(z_\pm,\theta)\geq 3\pi/(4p).$ By combining both cases, we have
		\begin{equation}\label{eq:zpmbds}
			\mathrm{dist}(\Theta_\alpha,z_\pm)\geq \pi/(4p),\quad \mathrm{dist}(\Theta_*,z_\pm)\geq \pi/(4p),\quad |\theta-z_\pm|\geq \pi/(4p).
		\end{equation}
		Combining the above statement with Lemma~\ref{lem:HatBound}, we obtain
		\begin{equation}\label{eq:maxintatzpm}
			\cM(\{z_-,z_+\})\geq \frac{p^2}{8}\left(\frac{\pi}{4p}\right)^3 = \frac{\pi^3}{512p}.
		\end{equation}
		Denote by $\gamma_v$ the union of the vertical components of the complex rectangle $\gamma$. By combining~\eqref{eq:maxintatzpm} with Lemma~\ref{lem:im_change}, we have $\cM(\gamma_v)\geq  {\pi^3}/({512p})$.
		
		We know that the integrand will oscillate along the horizontal components, which we denote by $\gamma_h$. We want to choose the length of $\gamma_v$ to be sufficiently large that {$\cM(\gamma_h)$ is bounded below, where $\gamma_h$ are the horizontal components of the rectangle $\gamma$}.
		From Lemma \ref{lem:lamupperbd} we have
		\[
		|\Lambda(\theta+\ri c,\alpha)|\geq \frac{\re^{|c|p}}{2} - \frac{3}{2},
		\]
		and it follows that if we choose $c:=\pm\ln(3+\pi^2/64)/p$, then
		\begin{equation}\label{eq:gammavbds}
			|\Lambda(\theta+\ri c,\alpha)|\geq\frac{\pi^2}{128},\quad\text{for }\theta\in\R.
		\end{equation}
		Thus we choose $\Im\gamma_h = \{-c,c\}$, where $c:=\pm\ln(3+\pi^2/64)/p$. Since $c\geq \pi/(4p)$, we have $\mathrm{dist}(\gamma_h,\theta)\geq \pi/(4p)$ and hence, from \eqref{eq:zpmbds}, $\mathrm{dist}(\gamma,\theta)\geq \pi/(4p)$. Combining this with \eqref{eq:gammavbds} gives
		\begin{equation}\label{eq:Mbound}
			\cM(\gamma)\geq\frac{\pi^3}{512p}.
		\end{equation}
		
		The final ingredient required in the bound~\eqref{eq:big_est} is a bound on the contour length $|\gamma|$. Recalling that the distance between the elements of $\Theta_*$ is $\pi/p$, and our choice of $z_\pm$, as stated above, we have $z_-\geq \theta_*-\pi/p-\pi/(2p)$, and $z_+\leq\theta_*+\pi/p$, so $|z_+-z_-|\leq5\pi/(4p)$. Thus we can bound the length of the rectangle as follows:
		\begin{equation}\label{eq:reclen}
			|\gamma|=2|z_+-z_-| + 4c \leq \frac{5\pi + 4\ln(3+\pi^2/64)}{4p}.
		\end{equation}
		Inserting \eqref{eq:Mbound} and \eqref{eq:reclen} into \eqref{eq:big_est} completes the proof.

		\texttt{}	%
	\end{proof}
	
	{We do not expect the rectangular choice of $\gamma$ to be optimal\revise{;} other choices of $\gamma$ may yield a smaller constant $C$ in Theorem \ref{thm:semidiscrete_error}.}
	
	\section{Computing the embedding coefficients}\label{sec:sampling}
	
	{We now focus on \ptwo. Intelligently choosing the canonical incident angles $\alpha_1,\ldots,\alpha_M$ such that Assumption \ref{ass:bexists} holds is a challenging problem. For standard bases of polynomials or trigonometric functions, \emph{unisolvence} theorems (see for e.g. \cite[\S2.4]{Da:75}) tell us that an $M$-dimensional function can be reconstructed uniquely from a set of $M$ distinct points. The problem \eqref{eq:sampling} may be interpreted similarly, as a reconstruction problem in the non-standard basis $\{\hat{D}(\cdot,\alpha_m)\}_{m=1}^M$; here we have no theoretical guarantees about reconstruction from $M$ distinct points. }Assuming $\bfb$ is unique and we can construct it, \eqref{eq:linconst} adds a further constraint that $\|\bfb\|_2$ cannot get too large (recall the discussion after Theorem \ref{thm:semidiscrete_error}). These difficulties are already significant before we consider that, in practice, we must work with an approximation to the matrix and right-hand side of \eqref{eq:sampling}.
	
	To address \ptwo, we consider two empirical approaches inspired by the philosophy of \emph{oversampling}. This approach has proven to be effective in cases with non-orthogonal bases which satisfy the \emph{frame} condition in \cite{AdHu:20}, and was observed to be effective for problems using high-frequency bases which \revise{do not satisfy the frame condition} in \cite{GiHeHuPa:20}. The idea here is to incorporate $\tilde{M}>M$ canonical far-field patterns into our algorithm, more than are strictly necessary according to \cite{Bi:06}. The larger we choose $\tilde{M}$, the more likely it is that there exists a subset of $\{\alpha_m\}_{m=1}^{\tilde{M}}$ which satisfies Assumption \ref{ass:bexists}. Likewise, we expect that if $\tilde{M}$ is large enough, there are many subsets which satisfy Assumption~\ref{ass:bexists}, and we are free to choose a subset which minimises $\|\bfb\|_2$. Thus, by introducing redundancy, we have more samples and basis functions available, reducing the risk of aliasing and the effects of ill-conditioning. Hence, we expect that this gives us a better chance of solving \ptwo.
	
	Suppose we use a \emph{sampling strategy}, which chooses $
	\{\alpha_m\}_{m\in\N}${ dense in }$\bbS$. 
	Then for $\tilde{M}>M$, we consider the \emph{oversampled system}
	\begin{align}\label{eq:oversampling}
		\tilde{A}\tilde\bfb=\tilde\bfd, \quad\text{where }\tilde{A}:=[\hat{D}(\alpha_m,\alpha_{m'})]_{m,m'=1}^{\tilde{M}},\quad\tilde\bfd=(-1)^{p+1}[\hat{D}(\alpha,\alpha_{m})]_{m=1}^{\tilde{M}}.
	\end{align}
	It is easy to see that for any $\tilde{M}>M$, we have $\mathrm{rank}(\tilde{A})\leq M$, {and we expect that there are multiple $\tilde{\bfb}$ which satisfy \eqref{eq:oversampling}.
		
		\revise{We have added redundancy to our problem, under the assumption that this increases the chance of finding a coefficient} vector which addresses \ptwo. We now consider two different strategies to select the \revise{coefficient} vector.}
	
	\subsection{\revise{Strategy One}}
	
	{For the first strategy, w}e assume that $\tilde{M}$ is chosen sufficiently large that there are multiple solutions $\tilde\bfb$ to \eqref{eq:oversampling}, each satisfying
	\begin{equation}\label{eq:exactos}
		\hat{D}(\cdot,\alpha)=\sum_{m=1}^{\tilde{M}}\tilde{b}_m(\alpha)\hat{D}(\cdot,\alpha).
	\end{equation}
	When multiple solutions exist to \eqref{eq:oversampling}, the system is under-determined, and Algorithm~\ref{alg:SVD} (see, e.g., \cite{GoLo:13}) with $\delta=0$ provides a pseudo-inverse which will find the solution with minimal norm - this is consistent with our aim to address \ptwo.
	
	\begin{algorithm}
		\caption{Pseudo-inverse via truncated singular value decomposition}\label{alg:SVD}
		\begin{algorithmic}[1]
			\STATE{\textbf{Inputs:} $X\in\C^{\tilde{M},\tilde{M}}$, $\delta>0$}
				\STATE{Compute the singular value decomposition,
					\[
					X = U\Sigma V^*
					\]}
				\STATE{Denote by $\sigma_m$ the $m$th entry of $\Sigma$. Define $\Sigma^\dagger$ as the diagonal matrix with entries
					\[
					\sigma_m^\dagger\gets\left\{\begin{array}{ll}
						1/\sigma_m&\text{if }\sigma_m>\delta,\\
						0&\text{if }\sigma_m\leq\delta
					\end{array}
					\right.,\quad m=1,\ldots,\tilde{M}
					\]}
				\RETURN{Return pseudo-inverse $X^{\dagger}\gets V\Sigma^\dagger U^*$}
		\end{algorithmic}
	\end{algorithm}
	
	However, we must consider the added implications of working with numerical approximations. Hence, we define another problem:
	\begin{equation}\label{eq:approxoversampling}
		\tilde{A}_N\tilde\bfb_N\approx\tilde\bfd_N, \quad\text{where }\tilde{A}_N:=[\hat{D}_N(\alpha_m,\alpha_{m'})]_{m,m'=1}^{\tilde{M}}\approx\tilde{A},\quad\tilde{\bfd}_N=(-1)^{p+1}[\hat{D}_N(\alpha,\alpha_{m})]_{m=1}^{\tilde{M}}\approx\tilde{\bfd}.
	\end{equation}
	To determine $\tilde{\bfb}_N$, we will use Algorithm \ref{alg:SVD}. The following lemma, simply an application of \cite[Lemma~3.3]{CoHuMaWe:20}, describes how the error is balanced between the residual and \revise{coefficient} norms.
	
	\begin{lemma}\label{lem:svdbd}
		Suppose that the pseudo-inverse $\tilde{A}_N^{\dagger}$ is obtained using Algorithm \ref{alg:SVD} with matrix $\tilde{A}_N$ and threshold $\delta$. Then
		\begin{equation}\label{eq:tbN_def}
			\tilde{\bfb}_{N}:=\tilde{A}_N^\dagger\tilde{\bfd}_N,
		\end{equation}
		satisfies
		\[
		\|\tilde{\bfd}_{N}-\tilde{A}_N\tilde{\bfb}_N\|_2
		\leq\inf_{\bfv\in\C^{\tilde{M}}}
		\left\{\|\tilde\bfd_N-\tilde{A}_N\bfv\|_2 + \delta\|\bfv\|_2 \right\}.
		\]
	\end{lemma}
	
	Lemma~\ref{lem:svdbd} informs us how to choose \revise{$\delta$} so as to balance the residual error with the \revise{coefficient} norm. A small choice of $\delta$ means that the residual error is relatively small, whereas a large choice of $\delta$ means that the \revise{coefficient} norm is relatively small. A balance of both is desirable. When balancing these two, we first consider that for fixed $N$, even if $\|\tilde{\bfd}_{N}-\tilde{A}_N\tilde{\bfb}_N\|_2\to0$ for $\delta\to0$, this does not imply $\|\tilde{\bfb}_{N}-\tilde{\bfb}\|_2\to0$, for any $\tilde{\bfb}$ solving \eqref{eq:exactos}. This is simply due to the problems \eqref{eq:oversampling} and \eqref{eq:approxoversampling} having different solutions. Obviously, as $N\to\infty$, we expect $\|\tilde{\bfb}_{N}-\tilde{\bfb}\|_2$ to be small, but it can only be made \emph{so small} by decreasing $\delta$. There will exist a threshold as $\delta\to0$, beyond which there is no practical advantage in decreasing $\delta$ any further. Moreover, choosing $\delta$ too small is a disadvantage, because the norm $\|\tilde{\bfb}_{N}\|_2$ is permitted to become very large as $\delta\to0$, which is inconsistent with our aim to address \ptwo. This threshold will depend on the accuracy of $D\approx D_N$, but can be estimated using the following result.
	
	\begin{lemma}\label{lem:resbound}
		For $\tilde{A}_N$ and $\tilde\bfd_N$ of the approximate system \eqref{eq:approxoversampling}, $\tilde{\bfb}$ of the exact system \eqref{eq:oversampling} and $\|\bfepsilon\|_2$ as in Theorem \ref{thm:semidiscrete_error},
		\[
		\inf_{\bfv\in\C^{\tilde{M}}}
		\left\{\|\tilde\bfd_N-\tilde{A}_N{\bfv}_N\|_2 \right\}
		\leq \|\tilde\bfd_N-\tilde{A}_N\tilde{\bfb}\|_2\leq
		2(\tilde{M}\|\tilde{\bfb}\|_2+1)\|\bfepsilon\|_2.
		\]
	\end{lemma}
	\begin{proof}
		The infimum is taken over all $\bfv\in\C^{\tilde{M}}$, we proceed by choosing $\bfv=\tilde{\bfb}$. We have
		\begin{align*}
			\tilde{\bfd}_N-\tilde{A}_N\tilde{\bfb} &= \tilde{\bfd}_N + (\tilde{\bfd}-\tilde{\bfd})-\left[\tilde{A}_N+(\tilde{A}-\tilde{A})\right]\tilde{\bfb} \\
			&=\left[\tilde{\bfd}_N-\tilde{\bfd} - (\tilde{A}_N-\tilde{A})\tilde{\bfb} \right]
			+\left[\tilde{\bfd}-\tilde{A}\tilde{\bfb}\right]\\
			&=\tilde{\bfd}_N-\tilde{\bfd} - (\tilde{A}_N-\tilde{A})\tilde{\bfb},
		\end{align*}
		by \eqref{eq:oversampling}.  Elementary norm manipulation then gives
		\[
		\|\tilde{\bfd}_N-\tilde{A}_N\tilde{\bfb}\|_2\leq\|\tilde{\bfd}_N-\tilde{\bfd}\|_2 + \|\tilde{A}_N-\tilde{A}\|_F\|\tilde{\bfb}\|_2,
		\]
		where $\|\cdot\|_F$ denotes the Frobenius norm. The assertion follows by bounding each entry of $\tilde{\bfd}_N-\tilde{\bfd}$ and $\tilde{A}_N-\tilde{A}$, in terms of $\epsilon_m$ for $m=1,\ldots,\tilde{M}$. The factor of two follows because $|\Lambda|\leq2$.
	\end{proof}
	In some sense, $\tilde\bfb$ is the perfect approximation to the solution of \eqref{eq:tbN_def}, because there will be no error in our embedding coefficients. Lemma \ref{lem:resbound} quantifies that even with this \emph{perfect solution}, there will still be a residual error; therefore, we should not waste {too much} effort minimising the residual, especially if this comes at the cost of $\|\tilde{\bfb}_N\|_2$ growing. Assuming that $\|\tilde{\bfb}_N\|_2\approx\|\tilde{\bfb}\|_2$, Lemmas \ref{lem:svdbd} and \ref{lem:resbound} suggest that to balance the residual error with the \revise{coefficient} norm, a sensible choice is
	\begin{equation}\label{eq:deltaguess}
		\delta \geq\tilde{M}\|\bfepsilon\|_2.
	\end{equation}
	This will minimise the \revise{coefficient} norm $\|\bfb\|_2$, subject to the constraint that the residual error is no smaller than necessary. In \S\ref{sec:osexps}, we experiment with different $\tilde{M}$ and $\delta$, providing solid numerical evidence that this approach addresses \ptwo.
	
	\subsection{Strategy Two}\label{sec:strategy2}
	
	In Strategy One, we addressed {\ptwo} by increasing the number of canonical far-field patterns in the embedding formula from $M$ to $\tilde{M}$, minimising the least-squares error. Strategy Two considers $\tilde{M}$ canonical far-field patterns only as a preprocessing step, before restricting to a subset of $M$ indices $\calI\subset\{1,\ldots,\tilde{M}\}$ which are in some sense optimal. Then in the embedding formula, we use the canonical incident angles $\{\alpha_m\}_{m\in\calI}$, thus choosing $A$ (of \eqref{eq:sampling}) as a square submatrix of $\tilde{A}$, keeping only the rows and columns in the index set $\calI$. \revise{Strategy Two} aims to choose this submatrix $A$ such that {\ptwo} is addressed.
	
	To achieve this, we use Algorithm \ref{alg:greedy}, initially proposed in \cite{BuGo:65}. This greedy algorithm is designed for problems which require a subset of matrix columns which maximises volume and equivalently (see \cite{HoPa:92}), minimises condition number. Although these are NP-hard optimisation problems, this algorithm achieves near-optimal results \cite{CiMa:09}.
	
	\begin{algorithm}
		\caption{Column subset selection}\label{alg:greedy}
		\begin{algorithmic}[1]
			\STATE{\textbf{Inputs:} $\tilde{A}=[\mathbf{a}_1|\ldots|\mathbf{a}_{\tilde{M}}]\in\C^{\tilde{M}\times\tilde{M}}$, $M<\tilde{M}$}
			\STATE{Define an empty array }$\calI\gets\{\}$
			\WHILE{$\calI$ contains fewer than $M$ elements}
			\STATE{Assign \[m^*\gets\argmax_{m=1,\ldots,\tilde{M}}\{|\mathbf{a}_m|\}\]}
			\STATE{Update $\calI\gets\calI\cup m^*$}
			\STATE{Update each vector}\[
			\mathbf{a}_m\gets\mathbf{a}_m - \mathbf{a}_{m^*}\frac{\left<\mathbf{a}_m,\mathbf{a}_{m^*}\right>}{\left<\mathbf{a}_{m^*},\mathbf{a}_{m^*}\right>}, \quad{m=1,\ldots,\tilde{M}}
			\]
			\ENDWHILE
			\RETURN{$\calI$}
			
		\end{algorithmic}
	\end{algorithm}
	
	Informally, at each iteration of the \textbf{while} loop, Algorithm \ref{alg:greedy} chooses the column vector which is (in some sense) the most orthogonal to the columns which have already been logged in the array $\calI$, via the same computation used in Gram-Schmidt orthogonalisation. 
	
	To the best knowledge of the authors, current theoretical results for Algorithm \ref{alg:greedy} apply to \revise{column subset selection, whereas our application requires us to choose a subset of columns and rows}. Despite the lack of available theory, our numerical experiments of \S\ref{sec:numerics} suggest that Strategy Two is highly effective in practice, provided $\tilde{M}$ is chosen sufficiently large; our experiments suggest that $\tilde{M}=\lceil 3M/2\rceil$ is typically more than sufficient. In terms of accuracy, we observe that it outperforms Strategy One at high frequencies. Moreover, considering that most CPU time is spent on far-field evaluations $D_N(\theta,\alpha)$, Strategy Two is more computationally efficient, by a factor of roughly $M/\tilde{M}$.

	\section{\revise{The algorithm}}\label{sec:mainalg}
	
	\begin{algorithm}
		\caption{Main routine}\label{alg:main}
		\begin{algorithmic}[1]
			\STATE\textbf{Inputs: $\theta,\alpha, \{D_N(\cdot,\alpha_m)\}_{m=1}^{\tilde{M}}, \{\frac{\partial}{\partial\theta}D_N(\cdot,\alpha_m)\}_{m=1}^{\tilde{M}}, \{\frac{\partial^2}{\partial\theta^2}D_N(\cdot,\alpha_m)\}_{m=1}^{\tilde{M}}, p$, $\texttt{Strategy}\in\{1,2\}$}
			\STATE\textbf{Adjustable parameters: $\delta=10^{-8},H=0.1,h=10^{-3}$}
			\STATE{\textbf{First pre-computation step:} Construct $\{\hat{D}_N(\cdot,\alpha_m)\}_{m=1}^{\tilde{M}}$ (and the first and second derivatives) using \eqref{eq:hattyboomboom}.}
			{\STATE{\textbf{Second pre-computation step - determining \revise{coefficient} vector.}
					Note that this step does not need to be repeated for future values of $(\theta,\alpha)\in\T$:
					\IF{$\texttt{Strategy}=1$}
					\STATE{Construct $\tilde{A}_{N}$ using \eqref{eq:approxoversampling}, then construct and store $\tilde{A}_{N}^\dagger$ using Algorithm \ref{alg:SVD} with inputs $\tilde{A}_{N}$ and $\delta$.}
					\ELSIF{$\texttt{Strategy}=2$}
					\STATE{Construct a submatrix $A_N$ using Algorithm \ref{alg:greedy} with inputs $\tilde{A}_N$ and $M$. Store $A_N^{-1}$.}
					\ENDIF}
				\STATE{Construct $\tilde\bfd_{N}$ using \eqref{eq:approxoversampling} and hence the \revise{coefficient} vector $\tilde\bfb_N$}}
			
			\IF{$|\theta-\theta_0|_{2\pi} > h$}
			\STATE Assign the naive approximation \eqref{eq:Dnaive}: $I\gets\Dnaive(\theta,\alpha)$.
			\IF{$|\theta-\theta_0|_{2\pi} \in[h,H)$}
			\STATE{A correction will be necessary - assign}
			\[
			I_{\mathrm{correction}}\gets\left\{\begin{array}{ll}
				\displaystyle\frac{1}{2\pi\ri}\oint_{\cR(\{\theta_0,\theta_0'\},h)}\frac{\rho_2(z)}{\Lambda(z,\alpha)(z-\theta)}\rd{z}		,\quad& |\theta_0-\theta_0'|_{2\pi}<h\\
				\displaystyle\sum_{\chi \in\{\theta_0,\theta_0'\}}\frac{ \sum_{m=1}^M  b_m(\alpha)\hat{D}_N(\chi,\alpha_m)}{p(\chi-\theta)\sin(p\chi)},\quad& |\theta_0-\theta_0'|_{2\pi}\in[h,H)\\
				\displaystyle\frac{ \sum_{m=1}^M b_m(\alpha)\hat{D}_N(\theta_0,\alpha_m)}{p(\theta_0-\theta)\sin(p\theta_0)},\quad& |\theta_0-\theta_0'|_{2\pi}>H\\
			\end{array}\right.
			\]
			\STATE{Update $I\gets I+I_{\mathrm{correction}}$}
			\ENDIF
			\ELSIF{$\theta=\theta_0=\theta_*$}
			\STATE{{Apply  L'H\^opital's rule (twice)}:
				\[
				I \gets \frac{\sum_{m=1}^{\tilde{M}}\tilde{b}_m(\alpha)\frac{\partial^2\hat{D}_N(\theta_0,\alpha_m)}{\partial\theta^2}}{-p^2\cos(p\theta)}
				\]}
			\ELSE \STATE{Risk of rounding errors, use polynomial interpolation at the poles.}
			\STATE{Construct $\rho_2$, interpolating as described in \S\ref{sec:FPA}.}
			\[
			I\gets\left\{\begin{array}{ll}
				\displaystyle\frac{1}{2\pi\ri}\oint_{\cR(\{\theta,\theta_0,\theta_0'\},h)}\frac{\rho_2(z)}{\Lambda(z,\alpha)(z-\theta)}\rd{z}	,\quad& |\theta_0-\theta_0'|_{2\pi}<h\\
				\displaystyle\frac{1}{2\pi\ri}\oint_{\cR(\{\theta,\theta_0\},h)}\frac{\rho_2(z)}{\Lambda(z,\alpha)(z-\theta)}\rd{z} + \displaystyle\frac{ \sum_{m=1}^M b_m(\alpha)\hat{D}_N(\theta_0',\alpha_m)}{p(\theta_0'-\theta)\sin(p\theta_0')},\quad& |\theta_0-\theta_0'|_{2\pi}\in[h,H)\\
				\displaystyle\frac{1}{2\pi\ri}\oint_{\cR(\{\theta,\theta_0\},h)}\frac{\rho_2(z)}{\Lambda(z,\alpha)(z-\theta)}\rd{z},\quad& |\theta_0-\theta_0'|_{2\pi}>H\\
			\end{array}\right.
			\]
			\ENDIF
			\STATE \textbf{Output:} $I$
		\end{algorithmic}
	\end{algorithm}
	
	The algorithm {we use for the numerical experiments in \S\ref{sec:numerics} }is based upon the ideas presented in \S\ref{sec:intro}--\S\ref{sec:sampling}, with some modifications in order to minimise unnecessary \revise{flops (floating point operations)}. It can be seen from Figure \ref{fig:illcond} that for some values of $\theta$ the naive approximation \eqref{eq:Dnaive} is sufficient. With this in mind, we introduce a threshold $H>0$, such that if $|\theta-\theta_0|_{2\pi}<H$, then we consider it necessary to correct the naive approximation in some way, otherwise we simply use~$\eqref{eq:Dnaive}$. When a correction is considered necessary, all of the relevant integrals and sums are based on equations \eqref{eq:res} and \eqref{eq:Cauchyrho3}, with the following exception:  the residue $\theta_0'$ is excluded from the sum \eqref{eq:res} when $|\theta_0-\theta_0'|_{2\pi}\geq H$, as it is considered to have a negligible contribution.
	
	Similarly, in the vast majority of cases, there will be no issues with rounding errors, and we introduce a second threshold $h<H$, such that if $|\theta_0-\theta_0'|_{2\pi}<h$, we use the interpolation approach described in \S\ref{sec:FPA}. For these contour integrals, we use a rectangular contour $\cR(X,h)$, where $X$ is the set of poles which the integral encloses, chosen so that $\cR(X,h)$ is the smallest rectangle with $X$ in its interior, satisfying $\dist(X,\cR(X,h))=h$. These contour integrals are evaluated using a $20$-point Gaussian quadrature rule along each edge.
	
	Algorithm \ref{alg:main} summarises the key steps of our implementation, excluding details such as quadrature (discussed above), for brevity. Our algorithm has been implemented in Matlab and is available at \cite{git:REEF}. This implementation is intended to be a proof of \revise{concept --- developing} a streamlined software package is reserved for future work.
	
	A key input to Algorithm \ref{alg:main} is $\{D_N(\cdot,\alpha_m)\}_{m=1}^{\tilde{M}}$, which must be obtained by some numerical method. We use two different numerical methods (outlined below) to broaden the range of the following experiments. In both methods, we reformulate the problem \eqref{Helm}-\eqref{HelmSRC} as the standard first kind boundary integral equation, discretise using a boundary element method, approximating $\partial u/\partial n$ on the boundary of $\Omega$, solving using oversampled collocation, as described in \cite{GiHeHuPa:20}. The code used for both solvers is available at \cite{git:HNABEMLAB}.
	
	\subsection{Standard $hp$ Boundary Element Method}\label{sec:stdbem}
	
	The first solver is a standard $hp$ boundary element method (BEM, see, e.g., \cite{SaSc:11}). We use a piecewise polynomial approximation space, defined on a graded mesh. The mesh is chosen so that the largest mesh element is no more than $\pi/k$ (half a wavelength) long. Towards the corners of the polygon, we use a geometric grading with grading parameter $0.15$, and $2P$ layers, where $P$ is the maximal polynomial degree, reducing the polynomial degree linearly towards corners, as described in, e.g., \cite{GiChLaMo:21}.
	
	It is well known that standard BEMs such as this need to increase degrees of freedom like $O(k)$ to maintain accuracy as $k$ increases (see, e.g., \cite{SaSc:11}). Therefore, we use this method in examples with moderate to low $k$, where $\Omega$ is a polygon.
	
	\subsection{Hybrid Numerical-Asymptotic Boundary Element Method (HNA BEM)}\label{sec:hnabem}
	
	This method is a non-standard $hp$-BEM, where the basis consists of piecewise polynomials multiplied by $k$-dependent oscillatory functions; details are given in \cite{GiHeHuPa:20}. The mesh is graded towards the corners of the obstacle in the same way as for the standard BEM described above. The key difference here is that the mesh width is not $k$-dependent because the oscillations are resolved by the basis functions \cite{HeLaCh:15}. The numerical implementation of \cite{GiHeHuPa:20} is available at \cite{git:HNABEMLAB}, and {enjoys $k$-independent cost }for screen problems. Using this solver, we are able to experiment for extremely large $k$ for the case where $\Omega$ is a screen.
	
	\subsection{Default parameters}\label{sec:params}
	Through extensive experimentation, the following parameters were found to provide a good balance between efficiency and accuracy:
	\begin{itemize}
		\item For the thresholds introduced above, we choose $H=0.15$ and $h=0.01$.
		\item We oversample with $\tilde{M}=\lceil3M/2\rceil$.
		\item We use Strategy Two, unless explicitly stated otherwise.
		\item We choose $\tilde{M}$ canonical incident angles equispaced on $\bbS$, with $\alpha_1=0$.
		\item When evaluating the contour integrals over the rectangular contours $\cR(X,h)$, we use a $20$-point Gaussian quadrature rule along each edge of the rectangle.
		\item If Strategy One is chosen instead of the default Strategy Two, we \revise{use} $\delta=10^{-8}$.
	\end{itemize}
	Unless mentioned otherwise in the following experiments, it can be assumed that these parameters have been used. It appears that smaller values of $H$ are sufficient when $k$ is large, potentially reducing unnecessary \revise{flops}. We do not investigate this link here, and use a fixed value for all $k$.

	\section{Numerical examples and experiments}\label{sec:numerics}
	
	In this section, we present numerical experiments demonstrating the effectiveness \revise{of Algorithm \ref{alg:main}}.
	
	\subsection{Example applications of Algorithm \ref{alg:main}}\label{sec:basicegs}
	
	The first experiment we present compares the far-field pattern produced by the naive embedding approximation \eqref{eq:Dnaive} and Algorithm \ref{alg:main}.
	\subsubsection*{Right-angled isosceles triangle, $k=10$}
	We consider the far-field induced by a plane wave with angle $\alpha=5\pi/4$, $k=10$, and $\Omega$ the right-angled {isosceles} triangle with vertices $\bfP_1=(0,0)$, $\bfP_2=(0,1)$ and $\bfP_1=(1,0)$. By Definition \ref{def:RationalPolygons} we have $q_1=6,q_2=q_3=7,p=4$ and $M=17$. {We use the default parameters of \S\ref{sec:params} and} compute $D_N$ using the standard {$hp$-}BEM solver of \S\ref{sec:stdbem}, with a reference solution of $N_{\mathrm{ref}}=600$. Figure~\ref{fig:ratriglowup} shows the errors for both approximations for a range of $N$. The naive approximation is shown as a dotted line, and the output of Algorithm \ref{alg:main} is shown as a thick line. {The relative error is measured as
		\begin{equation}\label{eq:relerr1}
			\frac{\|\cD_N(\cdot,\alpha)-D_{N_{\mathrm{ref}}}(\cdot,\alpha)\|_{L^\infty(\bbS)}}{\|D_{N_{\mathrm{ref}}}(\cdot,\alpha)\|_{L^\infty(\bbS)}},
		\end{equation}
		where each norm is approximated with $1000$ equispaced points. This approximation} will \revise{converge exponentially} {as the number of quadrature points increases,} because the far-field is periodic {and \revise{entire}}, see, e.g., \cite{TreWe:14}.
	\begin{figure}
		\centering
		\includegraphics[width=0.7\linewidth]{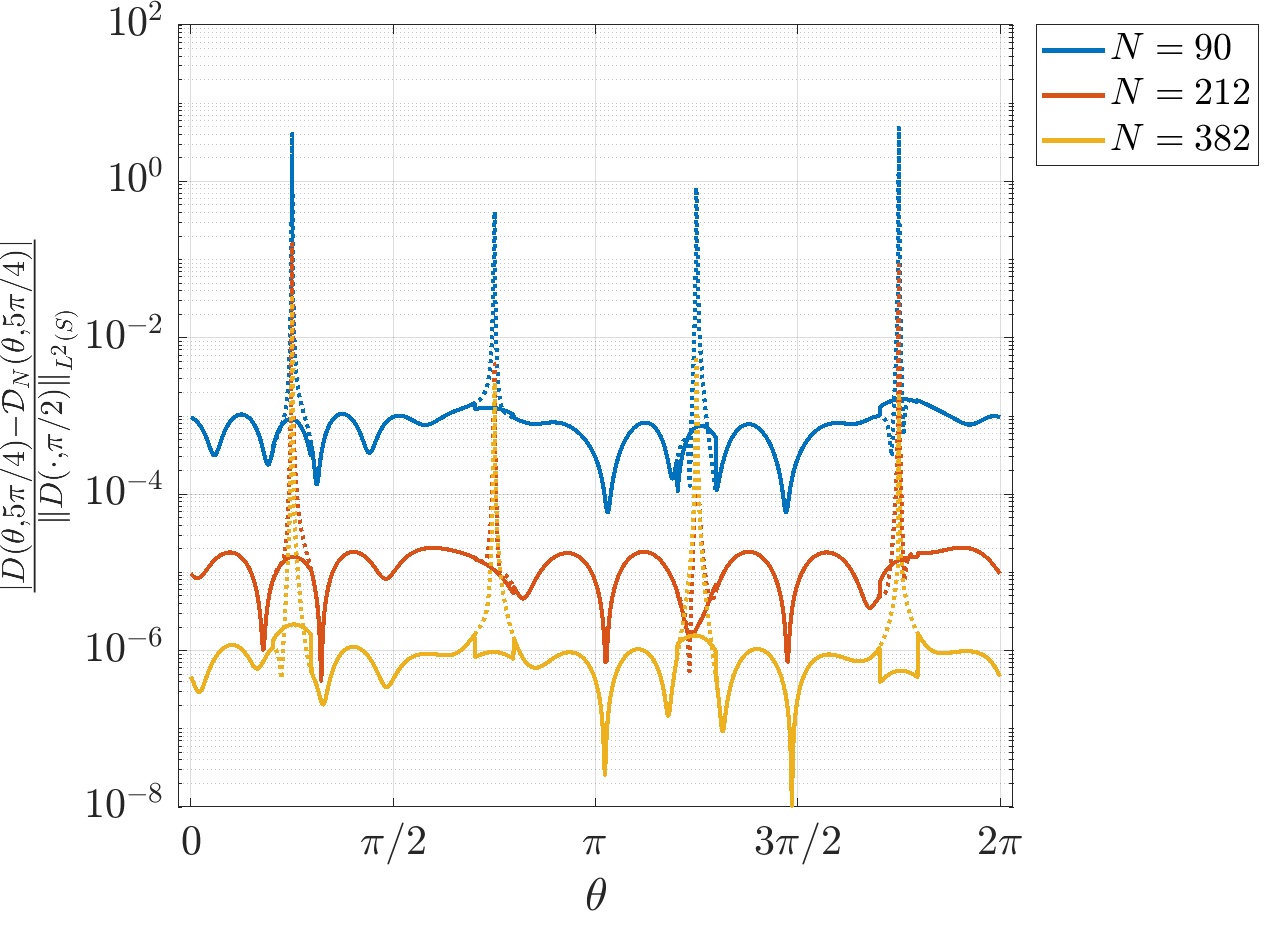}
		\caption{The errors in the naive approximation \eqref{eq:Dnaive} (dotted) and the output of Algorithm \ref{alg:main} (solid), for the right-angled triangle with $k=10$ and $\alpha=5\pi/4$. In all cases, the spikes of the naive approximation do not appear using our method.}
		\label{fig:ratriglowup}
	\end{figure}
	The unbounded error of the naive approach is clearly visible and remedied by our new approach, as predicted by Theorem~\ref{thm:semidiscrete_error}. {Discontinuous jumps in the error of our method are visible. These could be smoothed out by choosing a larger value of $H$, although it is not necessary to achieve a uniformly low error.}
	
	{Our method is most powerful when many incident angles are considered. Hence, for the next experiment, we consider one thousand equispaced $\alpha\in\bbS$. Figure \ref{fig:raerrsk10p4} shows $\Re[\cD_{212}]$ over $\T$ and the pointwise relative error, estimated using
		\[
		\frac{|\cD_N(\theta,\alpha)-D_{N_{\mathrm{ref}}}(\theta,\alpha)|}{|D_{N_{\mathrm{ref}}}(\theta,\alpha)|}.
		\]
		The relative error {is fairly evenly distributed, peaking at} around $10^{-4}${, with no visible spikes.}
		\begin{figure}
			\centering
			\includegraphics[width=0.47\linewidth]{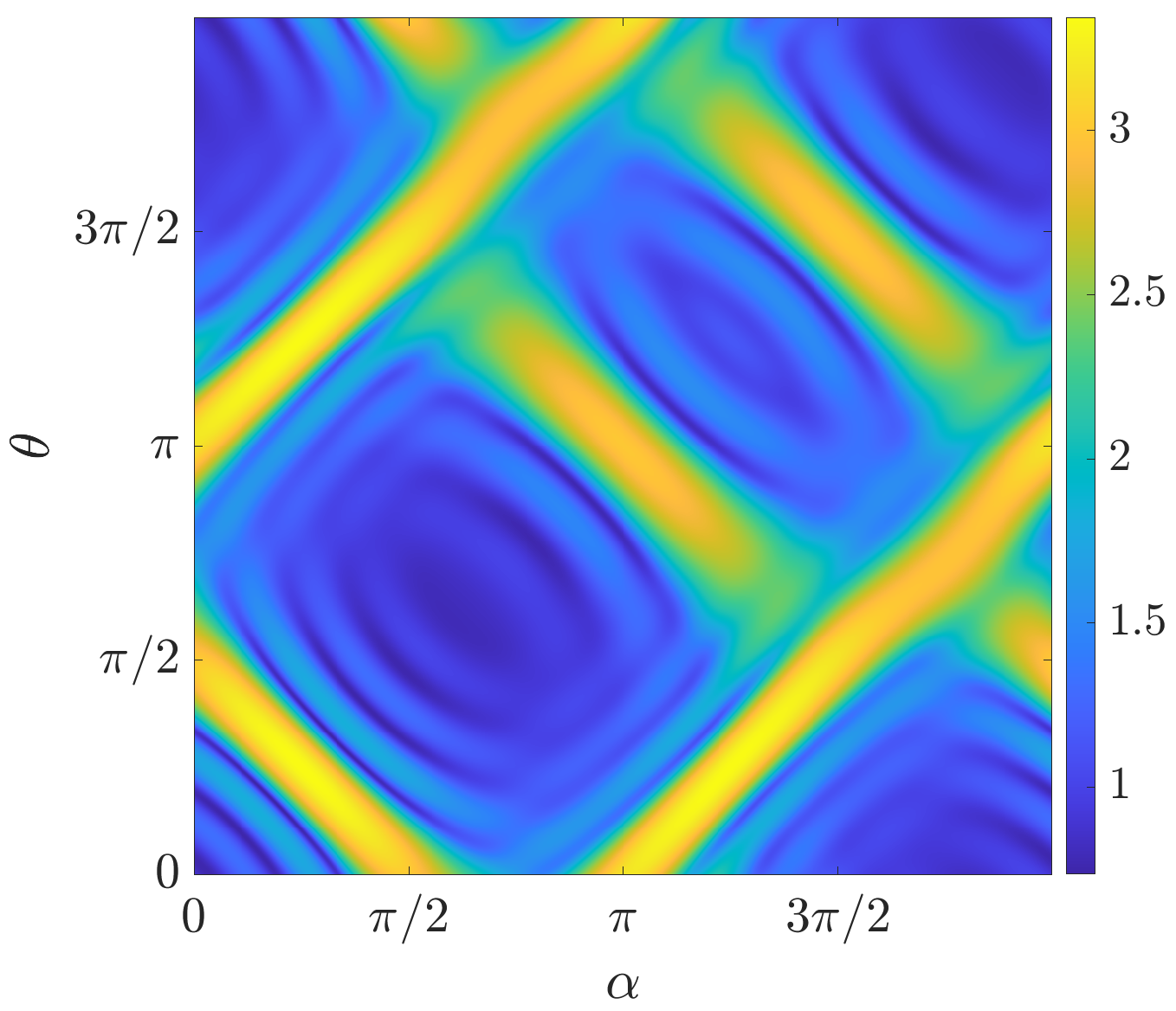}
			\includegraphics[width=0.48\linewidth]{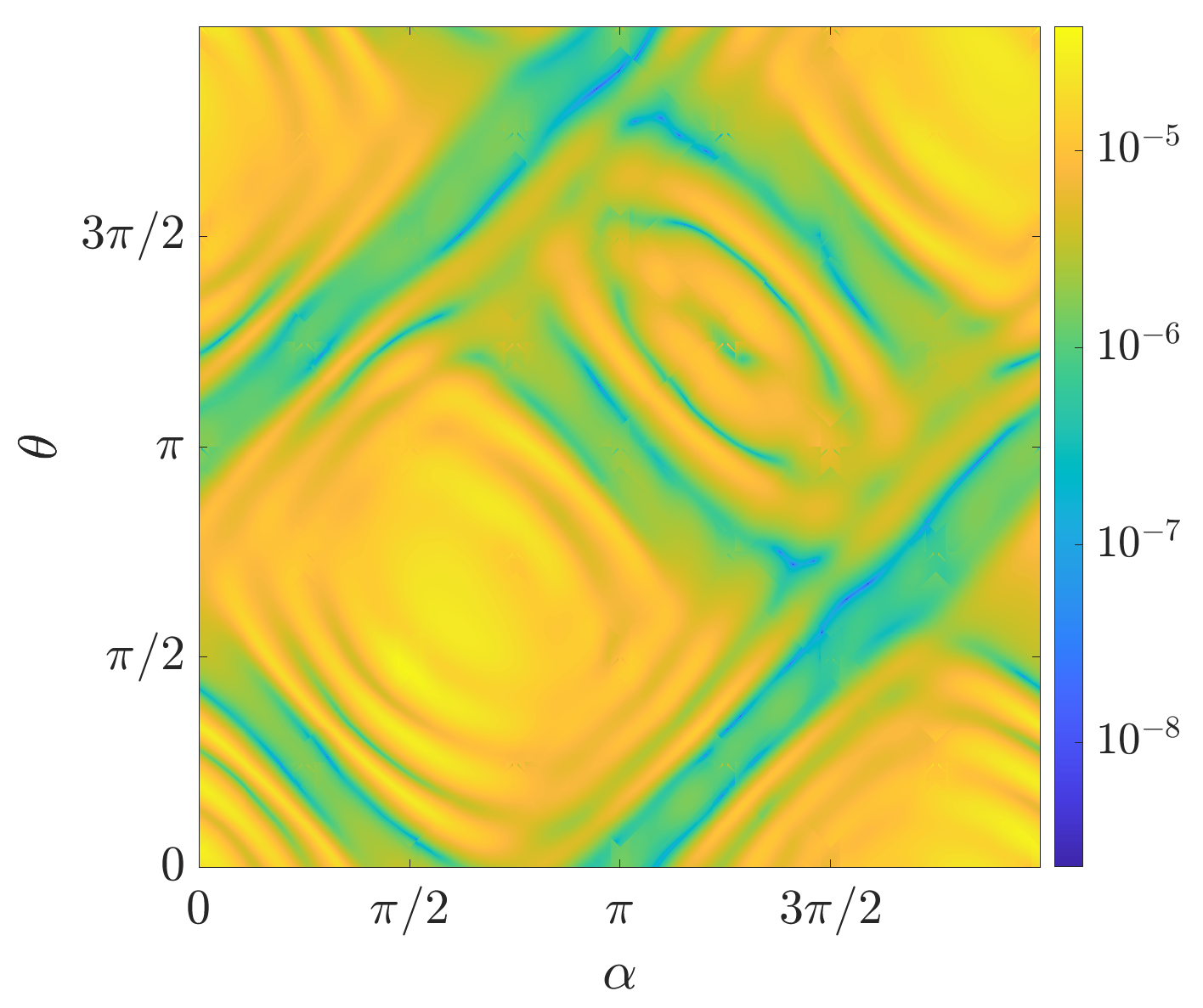}
			\caption{$\log|D(\theta,\alpha)|$ for right-angled {isosceles} triangle $k=10$ {(left)}, and corresponding pointwise relative error {(right)}.}
			\label{fig:raerrsk10p4}
		\end{figure}
	}
	
	\subsubsection*{Screen, $k=100$}
	{For a larger wavenumber $k=100$, we now consider the far-field induced by the screen $\Omega=[0,1]\times\{0\}$ and the incident angle $\alpha=\pi/4$. We use the HNA BEM solver described in \S\ref{sec:hnabem}, with a reference solution of $N_{\mathrm{ref}}=188$. The results are shown in Figure~\ref{fig:screenglowup}. Again, the naive approximation ({\eqref{eq:Dnaive}}, dotted lines) blows up at $\theta\in\Theta_\alpha$, which is fixed by our method. {The region of \revise{blow-up} appears to become narrower at higher frequencies, adding weight to our earlier claim in \S\ref{sec:params} that it may be possible to decrease $H$ as $k$ increases.}
		
		\begin{figure}
			\centering
			\includegraphics[width=0.47\linewidth]{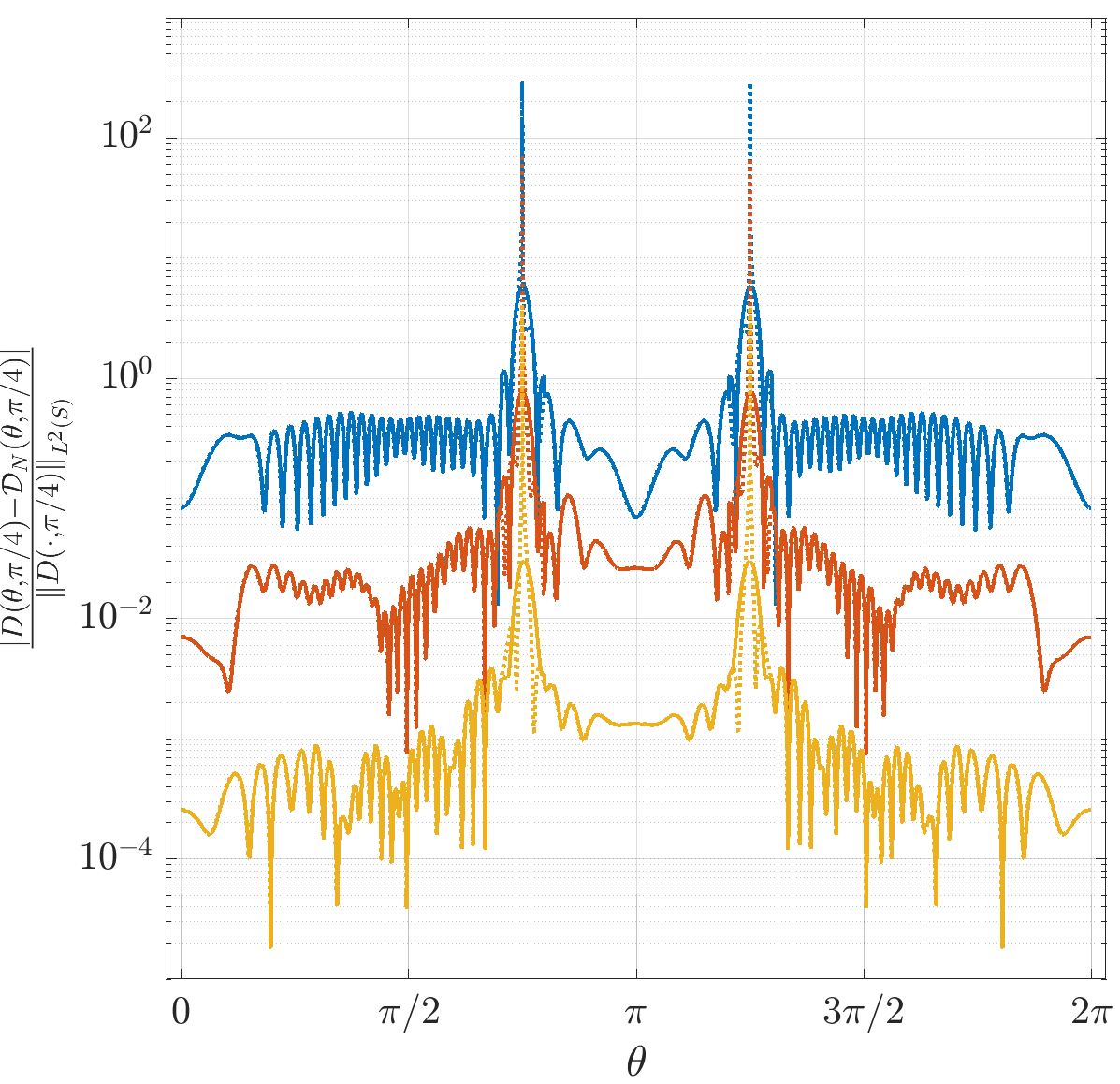}
			\includegraphics[width=0.52\linewidth]{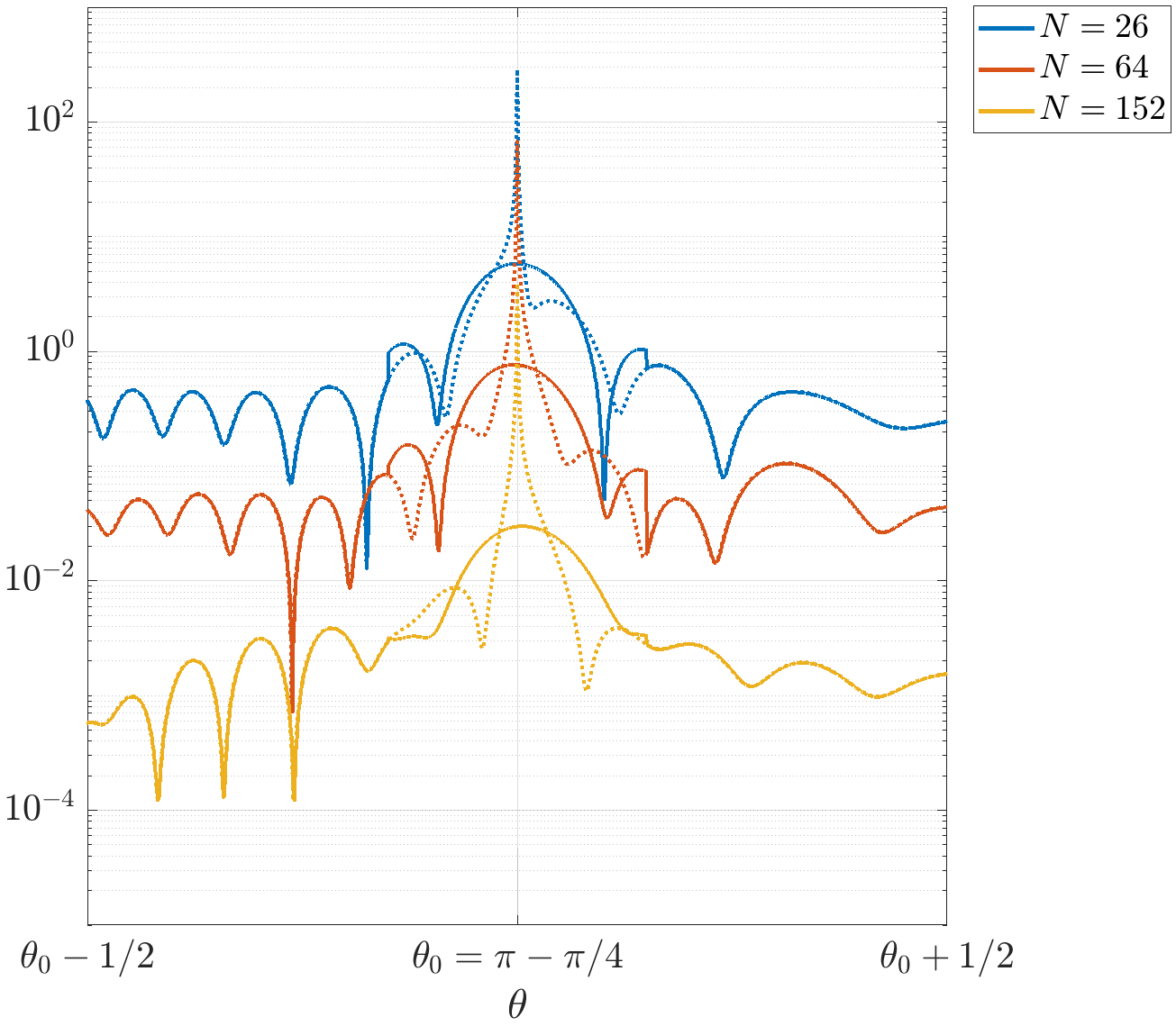}
			\caption{Left: the errors in the naive approximation \eqref{eq:Dnaive} (dashed) and the output of Algorithm \ref{alg:main} (solid), for the screen with $k=100$ and $\alpha=\pi/4$. Right: zoomed in around a point where \eqref{eq:Dnaive} breaks down, $\theta_0=\pi-\alpha$.}
			\label{fig:screenglowup}
		\end{figure}
		
		{As in Figure \ref{fig:raerrsk10p4}, we consider the approximation of $D(\theta,\alpha)$ for $1000$ values of $\alpha$ on the screen, with $N=152$ and $N_{\mathrm{ref}}=188$. The results are shown in Figure \ref{fig:screenerrsk10p4}. When compared against the low-frequency solution in Figure \ref{fig:raerrsk10p4}, we observe that the distribution of energy in the far-field has become more focused at $\theta=\pi\pm\alpha$, corresponding to the angles of reflection and propagation of the incident wave. This focusing is typical at higher frequencies.} %

	\begin{figure}
		\centering
		\includegraphics[width=0.475\linewidth]{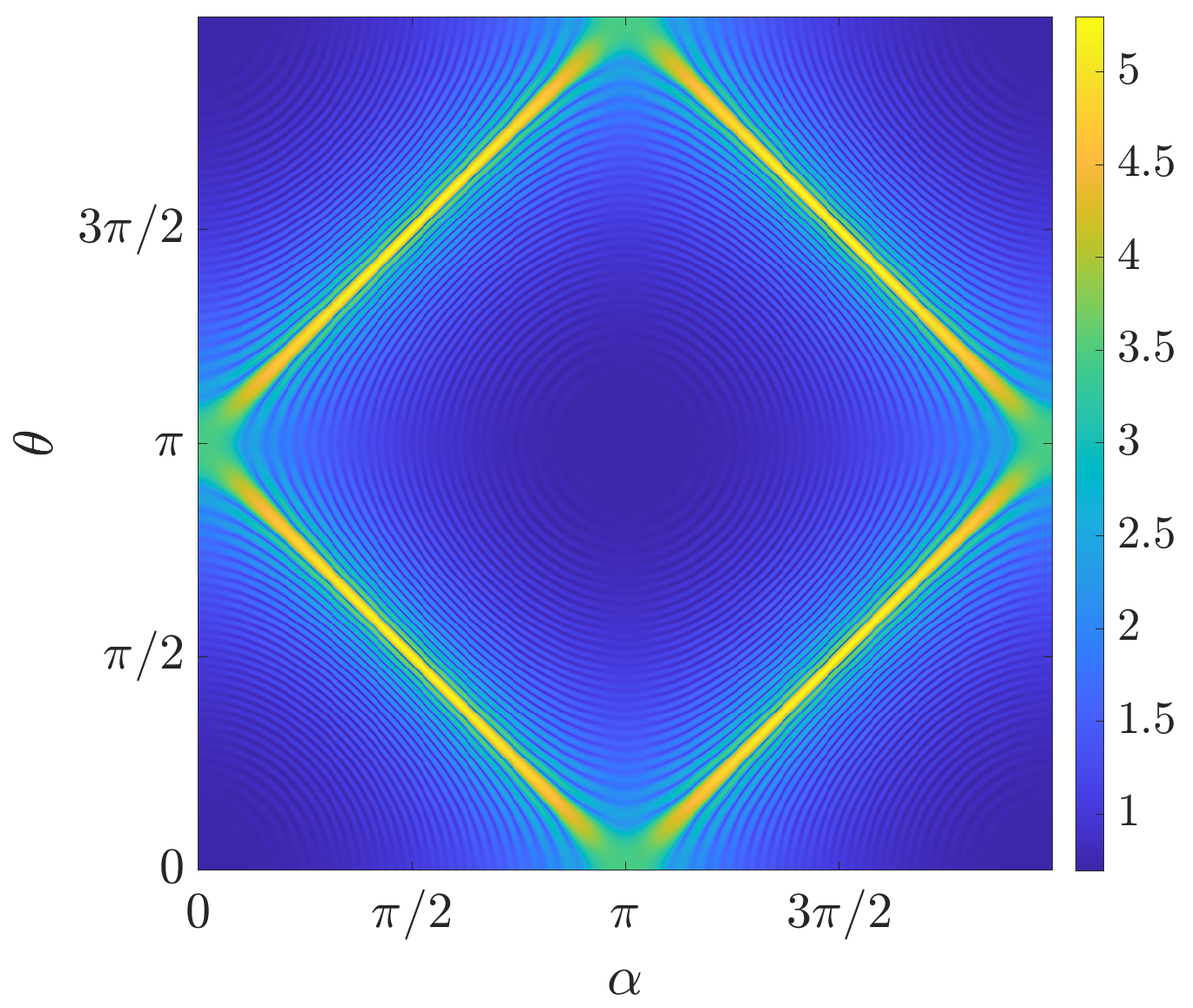}
		\includegraphics[width=0.475\linewidth]{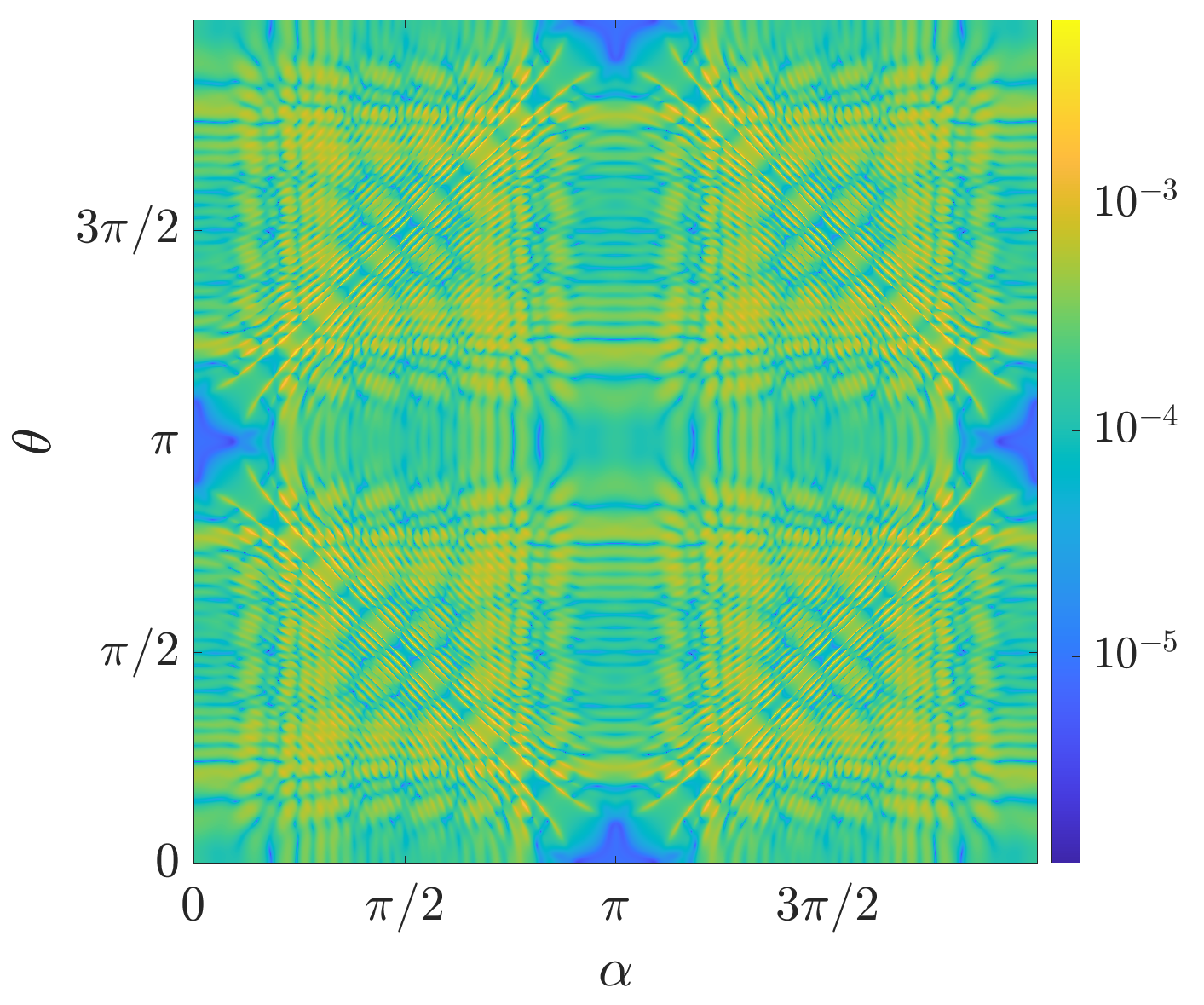}
		\caption{{$\log|D(\theta,\alpha)|$} for screen $k=100$ {(left)}, and corresponding pointwise relative error {(right)}.}
		\label{fig:screenerrsk10p4}
\end{figure}}

\subsection{Testing the strategies for \ptwo}\label{sec:osexps}

{Now we return to the two configurations identified in \S\ref{sec:embintro} for which the system \eqref{eq:sampling} is singular or ill-conditioned. These configurations motivated solutions for \ptwo, and two strategies were proposed in \S\ref{sec:sampling}. Here, we test these strategies for a range of oversampling parameters $\tilde{M}\geq M$ and truncation parameters $\delta$ (the latter only applies to Strategy One).}

\subsubsection*{Screen, $k=1000$}

{First, we consider the screen problem $\Omega=[0,1]\times\{0\}$, with canonical incident angles $\alpha_1=\frac{\pi}{2},\quad\alpha_2=\frac{3\pi}{2}$. Recalling the discussion in \S\ref{sec:embintro}, these incident angles correspond to $A_N$ equal to the $2\times2$ zero matrix. This issue will occur at all wavenumbers, but we consider $k=1000$. When oversampling, we ensure these two problematic incident angles are included, by considering the first $\tilde{M}$ incident angles of
	\[
	\alpha_1=\frac{\pi}{2},\quad\alpha_2=\frac{3\pi}{2},\quad\alpha_3=\pi,\quad\alpha_4=0,\quad\alpha_5=\frac{3\pi}{4},\quad\alpha_6=\frac{5\pi}{4}.
	\]
	Figure \ref{fig:mostestscreen} shows the effect of oversampling over the range $\tilde{M}=M=2$ up to $\tilde{M}=6$. We test Strategy One with truncation parameters $\delta\in\{10^{-12},10^{-8},10^{-4}\}$, and compare against Strategy Two.
	\begin{figure}[h]
		\centering
		\includegraphics[width=0.54\linewidth]{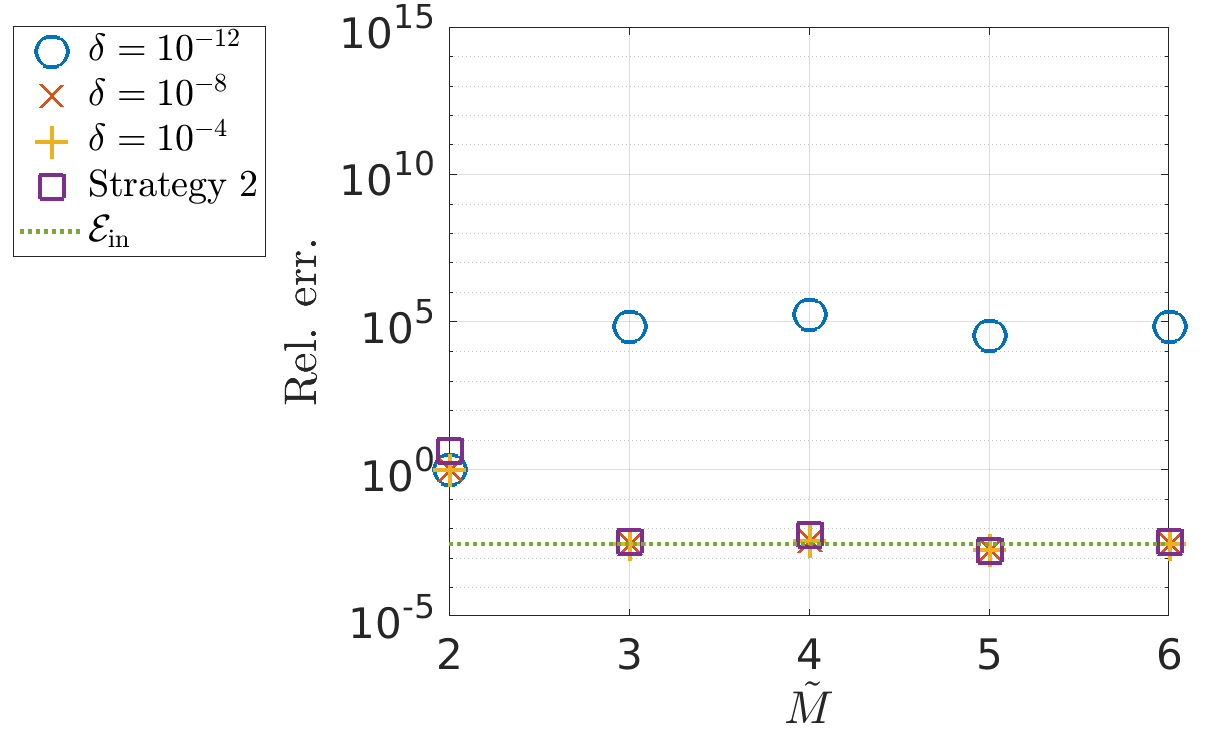}
		\includegraphics[width=0.44\linewidth]{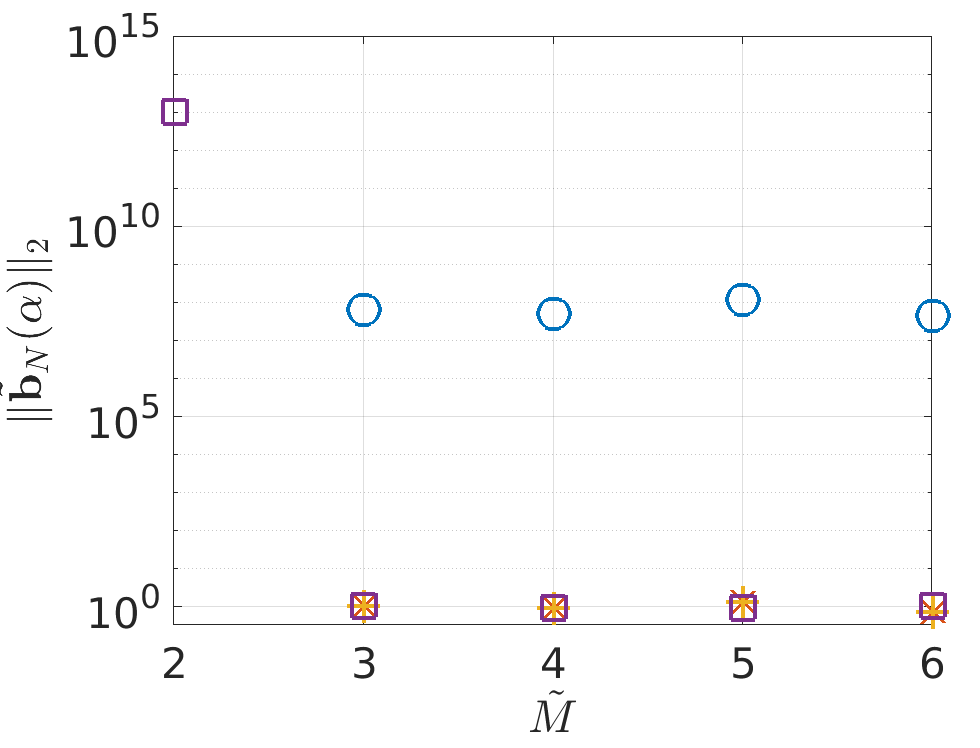}
		\caption{Error {and \revise{coefficient} norm,} varying $\delta$ and $\tilde{M}$ in Algorithm \ref{alg:main}, for the screen with $k=1000$.}
		\label{fig:mostestscreen}
	\end{figure}
	For the solver, we have used HNA BEM (\S\ref{sec:hnabem}), with $N=118$, and for the reference solution $N_{\mathrm{ref}}=188$. The relative error is measured using \eqref{eq:relerr1}
	and the `input error' is measured as
	\begin{equation}\label{eq:inputerr1}
		\errin:={\frac{\max_{m=1,\ldots\tilde{M}}\|D_N(\cdot,\alpha_m)-D_{N_{\mathrm{ref}}}(\cdot,\alpha_m)\|_{L^\infty(\bbS)}}{\max_{m=1,\ldots\tilde{M}}\|D_{N_{\mathrm{ref}}}(\cdot,\alpha_m)\|_{L^\infty(\bbS)}}}.
	\end{equation}
	The norms in \eqref{eq:relerr1} and \eqref{eq:inputerr1} are approximated using $1000$ equispaced samples.
	
	As expected, for $\tilde{M}=M=2$, i.e. without oversampling, the relative error is close to one. It should be noted that when $\tilde{M}=M$, Strategy Two does nothing, because the only valid submatrix is the full matrix. Therefore, the blow-up in the \revise{coefficient} norm and large relative error is what we expect. On the other hand, Strategy One responds by minimising the coefficient norm, which is zero.
	
	For $\tilde{M}\geq3$ Strategy Two performs well, and over the same range with  $\delta\in\{10^{-8},10^{-4}\}$, Strategy One performs well. The experiments show that $\delta=10^{-12}$ is too low for accurate results, which is consistent with the choice \eqref{eq:deltaguess}.
	
	\subsubsection*{Equilateral triangle, $k=10$}
	
	Next we consider the problem where $\Omega$ is the \revise{equilateral triangle with side length $\sqrt{12}/2$ (such that the vertices lie on the unit circle)}, with wavenumber $k=10$. Recalling Figure \ref{fig:badtricond25p8}, we observed that the set of canonical incident angles $\alpha_m = {2(m-1)\pi}/{12},$ for $m=1,\ldots,M=12$,
	caused $\cond(A_N)$ to blow up. We ensure these problematic incident angles are included, and oversample with up to an additional four angles, equispaced between consecutive angles in the above set, like so: $\alpha_m = {2\pi}/{24} + \alpha_{m-12}$, for $m=13,\ldots,16$. Now we use a standard $hp$-BEM, with $N=375$ and $N_{\mathrm{ref}}=591$, and as for the screen, the relative error is measured using \eqref{eq:relerr1} and the input error is measured using \eqref{eq:inputerr1}.
	
	As for the screen problem, choosing $\tilde{M}=M+1$ appears to be sufficient, and for this experiment all values of $\delta$ appear to work well, see Figure \ref{fig:mostesttri}. The coefficient norm is slightly increasing for $\delta=10^{-12},10^{-8}$ and $\tilde{M}=15,16$, but this does not affect the relative error.
	
	In both of these experiments and in others which are not reported, we have observed that Strategy Two performs {at least} as well as Strategy One{, provided we oversample sufficiently}. Considering that Strategy Two has numerous advantages (stated at the end of \S\ref{sec:strategy2}), we recommend this as the default choice.

	\begin{figure}
		\centering
		\includegraphics[width=0.54\linewidth]{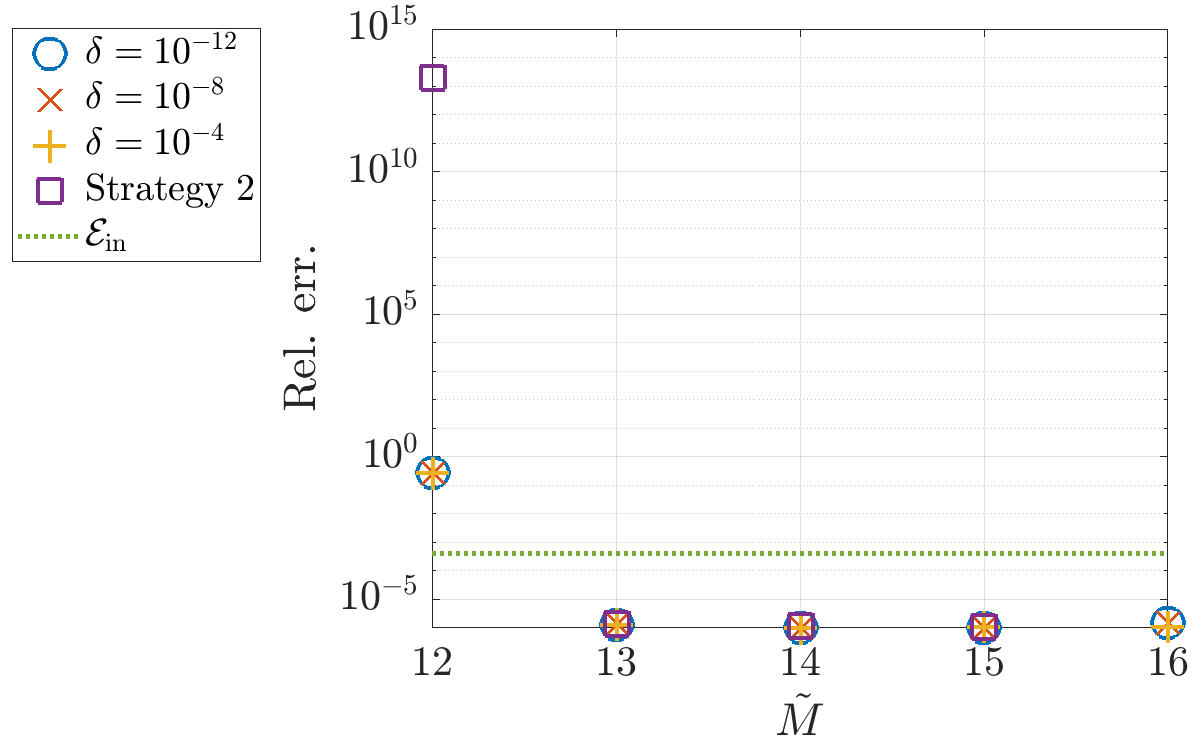}
		\includegraphics[width=0.43\linewidth]{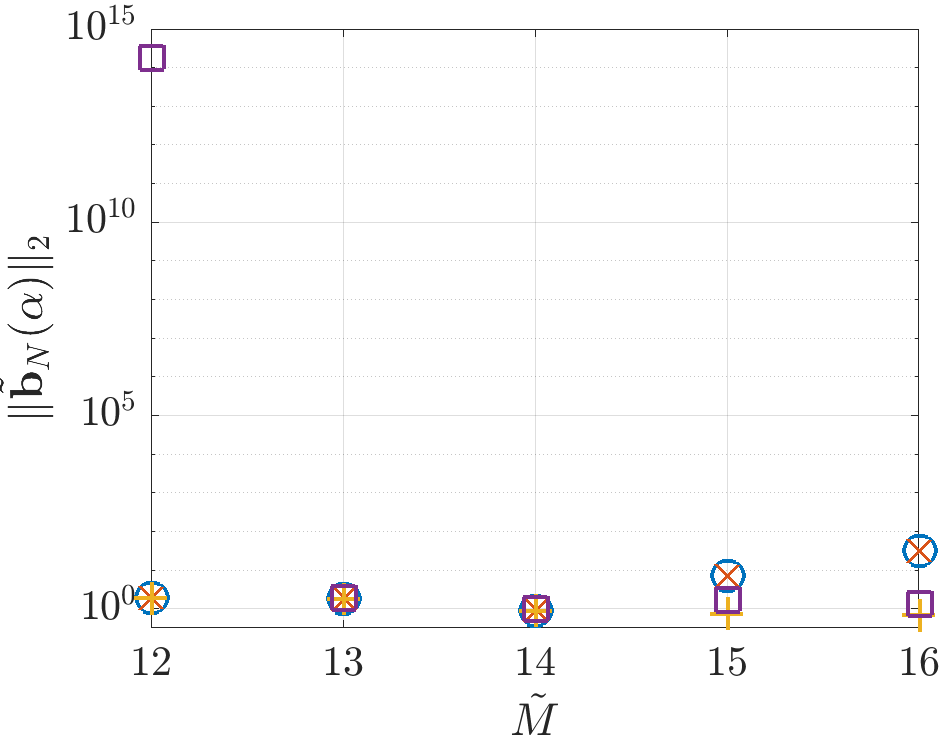}
		\caption{Error {and \revise{coefficient} norm,} varying $\delta$ and $\tilde{M}$ in Algorithm \ref{alg:main}, for the equilateral triangle with {$k=10$}.}
		\label{fig:mostesttri}
	\end{figure}
	
	\subsection{Conditioning estimates}\label{sec:condests}
	
	The constant $C$ in Theorem \ref{thm:semidiscrete_error} explains the relationship between input error (in the canonical far-field approximation) and output error (of our embedding formula \eqref{eq:CauchyFFN}). In this sense, $C$ describes the conditioning of our embedding formula. Now we consider {another quantity to measure conditioning}, measured in terms of the ratio of $\errout$ and $\errin$, where
	\[
	\errout:=\frac{\|\cD_N-D_{N_\mathrm{ref}}\|_{L^\infty(\T)}}{\|D_{N_\mathrm{ref}}\|_{L^\infty(\T)}}.
	\]
	We {approximate the norms using a tensor product} trapezoidal rule with $1000\times1000$ points (again, this converges exponentially by arguments in, e.g., \cite{TreWe:14}). {Here we investigate the relationship between the input and output errors, over a range of wavenumbers, scatterers and solvers.}%
	
	\subsubsection*{Regular polygons, low wavenumbers}
	
	\begin{table}
		\centering
		\begin{tabular}{|c|c|c||S[table-format=1.2e-1]|S[table-format=1.2e-1]|S[table-format=1.2e-1]|S[table-format=1.2e-1]|}
			\hline
			$k$ & $\Omega$ & $N$ & {$\errin$} & {$\errout$} & {$\errout/\errin$} & {$\cond(A_N)$}\\
			\hline
			5 &  triangle & 78 & 4.3e-04 &1.0e-02 &2.4e+01 &4.3e+00\\
			&  triangle & 192 & 3.6e-06 &7.5e-05 &2.1e+01 &4.3e+00\\
			&  triangle & 354 & 9.7e-08 &1.3e-06 &1.3e+01 &4.3e+00\\
			&	 square &	104&	1.0E-04&	3.5E-02&	3.4E+02&	2.1E+00\\
			&	 square &	256&	1.2E-06&	2.8E-04&	2.4E+02&	2.1E+00\\
			&	 square &	472&	4.0E-08&	1.1E-05&	2.8E+02&	2.1E+00\\
			&	 pentagon& 	130&	4.1E-05&	2.3E-03&	5.6E+01&	1.2E+05\\
			&	 pentagon &	320&	5.1E-07&	2.6E-05&	5.0E+01&	1.2E+05\\
			&	 pentagon &	590&	1.9E-08&	2.1E-07&	1.1E+01&	1.2E+05\\
			\hline
			25 & triangle & 168 & 9.8e-04 & 2.7e-02 & 2.8e+01 & 2.8e+01\\
			& triangle & 342 & 4.3e-06 & 4.0e-04 & 9.3e+01 & 1.4e+01\\
			& triangle & 564 & 8.3e-08 &1.1e-05 &1.3e+02 &1.4e+01\\
			&	 square &	200&	6.8E-04&	7.1E-01&	1.1E+03&	1.3E+01\\
			&	 square &	416&	3.4E-06&	3.9E-03&	1.1E+03&	1.3E+01\\
			&	 square &	696&	5.6E-08&	2.5E-05&	4.6E+02&	1.3E+01\\
			&	 pentagon& 	235&	4.8E-04&	1.9E+00&	4.0E+03&	3.7E+02\\
			&	 pentagon &	495&	2.1E-06&	7.6E-03&	3.6E+03&	6.2E+02\\
			&	 pentagon &	835&	4.4E-08&	5.8E-05&	1.3E+03&	6.2E+02\\
			
			\hline
		\end{tabular}
		\caption{{$L^\infty$ input and output errors for a range of regular polygons. \revise{Error values and condition numbers are reported to two significant figures.}}}\label{tab:stdbem2}
	\end{table}
	
	Low-frequency results for \revise{when $\Omega$ is a regular polygons with vertices positioned on the unit circle}, where $D_N$ is the standard BEM solver (\S\ref{sec:stdbem}), are given in {Table \ref{tab:stdbem2}. The same experiment was performed for Strategy One, with similar results.} For the triangle, $N_\mathrm{ref}=942$, for the square, $N_\mathrm{ref}=1400$, and for the pentagon, $N_\mathrm{ref}=1660$. {In all cases, we observe convergence as $N$ increases, and our method can achieve a high accuracy for all incident angles with a relatively low $N$. There appears to be no obvious rule for predicting $\errout/\errin$.
	}

		\subsubsection*{Screen, high wavenumbers}
		\begin{table}
			\centering
			\begin{tabular}{|c|c||S[table-format=1.2e-1]|S[table-format=1.2e-1]|S[table-format=1.2e-1]|c|}
				\hline
				$k$ & $N$ & {$\errin$} & {$\errout$} & {$\errout/\errin$} & {$\cond(A_N)$}\\
				\hline
				500	&	44	&	4.2E-04		&	8.4E-01	&	2.0E+03	&	2.8	\\
				&	90	&	4.3E-05		&	2.5E-01	&	5.7E+03	&	2.8	\\
				&	152	&	1.0E-05	&	6.0E-02	&	6.0E+03	&	2.8	\\
				&	230	&	2.9E-06		&	1.6E-02	&	5.6E+03	&	2.8	\\
				\hline
				1000	&	44	&	2.4E-04	&	1.9E+00	&	8.1E+03	&	2.1	\\
				&	90	&	6.2E-05	&	5.8E-01	&	9.4E+03	&	2.2	\\
				&	152	&	4.8E-06	&	3.2E-02	&	6.7E+03	&	2.1	\\
				&	230	&	2.9E-06	&	5.7E-03	&	2.0E+03	&	2.1	\\
				\hline
				5000	&	44	&	9.6E-05	&	6.5E-01	&	6.7E+03	&	2.6	\\
				&	90	&	1.3E-05	&	4.7E-02	&	3.6E+03	&	2.3	\\
				&	152	&	2.6E-06	&	1.1E-02	&	4.1E+03	&	2.3	\\
				&	230	&	1.3E-06	&	1.9E-03	&	1.5E+03	&	2.3	\\
				\hline
				10000	&	44	&	8.2E-05	&	5.4E-01	&	6.6E+03	&	2.6	\\
				&	90	&	1.2E-05	&	7.4E-02	&	6.1E+03	&	2.9	\\
				&	152	&	1.9E-06	&	7.5E-03	&	3.9E+03	&	3.0	\\
				&	230	&	1.0E-06	&	2.1E-03	&	2.0E+03	&	3.0	\\
				
				\hline
			\end{tabular}
			\caption{{$L^\infty$ input and output errors, for large $k$ when $\Omega$ is a screen. \revise{Error values and condition numbers are reported to two significant figures.}}}\label{tab:hnabem}
		\end{table}

		High-frequency results on the screen, where $D_N$ is the HNA BEM of \S\ref{sec:hnabem}, are given in Table \ref{tab:hnabem}. Here $N_\mathrm{ref}=188$. Again, we observe convergence in each case as $N$ increases. {Convergence was observed to be much slower when the same experiments were run for Strategy One}. For {$N=230$, we observe {$\approx1\%$ error or less} for all wavenumbers tested}. This suggests that $N$ does not need to be very large to accurately represent the far-field pattern for all incident angles at high frequencies. This is a very encouraging result, suggesting that when paired with the HNA BEM, {the error and cost of our method remain} fixed for large $k$.
		
		\revise{
			\section{Conclusions and future work}\label{sec:future}
			
			Embedding formulae describe the fascinating theoretical connection between the far-field patterns induced by different incident plane waves. We have shown that with careful modifications, some of these formulae can be of practical use, significantly reducing the cost in numerical scattering models for two-dimensional sound-soft polygons.
			
			It is natural to ask if the techniques of this paper may be generalised to different scattering configurations. Focusing on two natural extensions, Table \ref{tab:context} places this current work within the context of necessary related results. The table is intended to highlight gaps elsewhere in the current scattering literature, which must be filled before the work of this paper can (possibly) be generalised.

			\begin{table}
				\begin{tabular}{|p{3cm}||p{3.25cm}|p{3.25cm}|p{3.25cm}|}
					\hline 
					Step & Sound-soft polygons 	& Sound-hard polygons  & Sound-soft polyhedra \\
					\hline\hline
					1. Edge Green Embedding formulae & Craster and Shanin \cite{KrSh:05}.	& Also in Craster and Shanin \cite{KrSh:05}.  & Some cases in \cite{SkCrShVa:10}, for example cubes. \\
					\hline
					2. Far-field Embedding formulae & Biggs \cite{Bi:06}. &  In \cite{Bi:06, Bi:16}, Biggs states that sound-hard formulae can be derived with minor modifications to sound-soft problem, but no results have yet been published. &  \\
					\hline
					3. Numerically robust modification & \textbf{This paper.} &  &  \\
					\hline
					4. Corresponding high-frequency solver & Screens: \cite{GiHeHuPa:20, git:HNABEMLAB}. Convex polygons: requires frequency-independent implementation of \cite{ChLa:07,HeLaMe:13} (work in progress).& Requires frequency-independent implementation of \cite{ChLaMo:12} (work in progress).&  Initial ideas were discussed in \cite[\S7.6]{ChGrLaSp:12}. Initial experiments on square screens were presented in \cite{hargreaves2015high}. No frequency-independent solver is available.\\
					\hline
				\end{tabular}\label{tab:context}
				\caption{Overview of existing literature on embedding formulae and high-frequency solvers for exterior scattering problems.}
			\end{table}
			
			It is clear from Table \ref{tab:context} that the main gap in the current literature is \emph{Step two} - embedding formulae \emph{in terms of far-field patterns.} In principle one could skip Step 2, applying the ideas of this paper to the embedding formulae of \cite{KrSh:05}, which also hold for sound-hard problems and contain $1/\Lambda(\theta,\alpha)$-type removable singularities; this is a possible area for future work. A similar approach may be possible for the three dimensional structures in \cite{SkCrShVa:10}. However, a key practical advantage of Step 2 is that (to the best knowledge of the authors) there are many existing solvers for computing far-field patterns, and far fewer for computing edge Green's functions.%
			
			We remark that Step 4 is not essential for embedding formulae to be of practical use; any problem which requires the far-field pattern induced for a large number of incident waves may enjoy a reduced computational cost using a numerically robust embedding formula, if one exists. We expect to have developed a frequency-independent solver for convex polygons in the near future, and we are excited to combine it with Algorithm \ref{alg:main}, and investigate the performance at high frequencies.
		}
		
		\section{Acknowledgements}
		
		The authors thank Nicholas Biggs, Abi Gopal, Stuart Hawkins, {Dave Hewett,} Daan Huybrechs, {Andrea Moiola}, Jennifer Scott and Marcus Webb for helpful conversations. \revise{We greatly appreciate the valuable comments and suggestions from Nick Trefethen and the anonymous referees.} AG is grateful for support from EPSRC grants EP/S01375X/1 and EP/V053868/1.
		
		\bibliographystyle{siam}
		\bibliography{refs}

\begin{thebibliography}{10}

\bibitem{AdHu:20}
{\sc B.~Adcock and D.~Huybrechs}, {\em Frames and numerical approximation {II}:
  {G}eneralized sampling}, J. Fourier Anal. Appl., 26 (2020).

\bibitem{AuKrTr:14}
{\sc A.~P. Austin, P.~Kravanja, and L.~N. Trefethen}, {\em Numerical algorithms
  based on analytic function values at roots of unity}, SIAM J. Numer. Anal.,
  52 (2014), pp.~1795--1821.

\bibitem{Bi:06}
{\sc N.~R.~T. Biggs}, {\em A new family of embedding formulae for diffraction
  by wedges and polygons}, Wave Motion, 43 (2006), pp.~517--528.

\bibitem{Bi:16}
\leavevmode\vrule height 2pt depth -1.6pt width 23pt, {\em Embedding formulae
  for scattering in a waveguide containing polygonal obstacles}, Quart. J.
  Mech. Appl. Math., 69 (2016), pp.~409--429.

\bibitem{BuGo:65}
{\sc P.~Businger and G.~H. Golub}, {\em Handbook series linear algebra.
  {L}inear least squares solutions by {H}ouseholder transformations}, Numer.
  Math., 7 (1965), pp.~269--276.

\bibitem{CiMa:09}
{\sc A.~\c{C}ivril and M.~Magdon-Ismail}, {\em On selecting a maximum volume
  sub-matrix of a matrix and related problems}, Theoret. Comput. Sci., 410
  (2009), pp.~4801--4811.

\bibitem{ChGrLaSp:12}
{\sc S.~N. Chandler-Wilde, I.~G. Graham, S.~Langdon, and E.~A. Spence}, {\em
  Numerical-asymptotic boundary integral methods in high-frequency acoustic
  scattering}, Acta Numer., 21 (2012), pp.~89--305.

\bibitem{ChLa:07}
{\sc S.~N. Chandler-Wilde and S.~Langdon}, {\em A {G}alerkin boundary element
  method for high frequency scattering by convex polygons}, SIAM J. Numer.
  Anal., 45 (2007), pp.~610--640.

\bibitem{ChLaMo:12}
{\sc S.~N. Chandler-Wilde, S.~Langdon, and M.~Mokgolele}, {\em A high frequency
  boundary element method for scattering by convex polygons with impedance
  boundary conditions}, Commun. Comput. Phys., 11 (2012), pp.~573--593.

\bibitem{CoKr:13}
{\sc D.~Colton and R.~Kress}, {\em Inverse acoustic and electromagnetic
  scattering theory}, vol.~93 of Applied Mathematical Sciences, Springer, New
  York, third~ed., 2013.

\bibitem{CoHuMaWe:20}
{\sc V.~Copp\'{e}, D.~Huybrechs, R.~Matthysen, and M.~Webb}, {\em The {AZ}
  algorithm for least squares systems with a known incomplete generalized
  inverse}, SIAM J. Matrix Anal. Appl., 41 (2020), pp.~1237--1259.

\bibitem{KrSh:05}
{\sc R.~V. Craster and A.~V. Shanin}, {\em Embedding formulae for diffraction
  by rational wedge and angular geometries}, Proc. R. Soc. Lond. Ser. A Math.
  Phys. Eng. Sci., 461 (2005), pp.~2227--2242.

\bibitem{Da:75}
{\sc P.~J. Davis}, {\em Interpolation and approximation}, Dover Publications,
  Inc., New York, 1975.

\bibitem{GaHa:09}
{\sc M.~Ganesh and S.~C. Hawkins}, {\em A far-field based {$T$}-matrix method
  for two dimensional obstacle scattering}, ANZIAM J., 51 (2009),
  pp.~C215--C230.

\bibitem{GaHaHi:12}
{\sc M.~Ganesh, S.~C. Hawkins, and R.~Hiptmair}, {\em Convergence analysis with
  parameter estimates for a reduced basis acoustic scattering {T}-matrix
  method}, IMA J. Numer. Anal., 32 (2012), pp.~1348--1374.

\bibitem{GaWe:90}
{\sc G.~C. Gaunaurd and M.~F. Werby}, {\em {Acoustic Resonance Scattering by
  Submerged Elastic Shells}}, Appl. Mech. Rev., 43 (1990), pp.~171--208.

\bibitem{git:HNABEMLAB}
{\sc A.~Gibbs}, {\em {HNABEMLAB}}, https://github.com/AndrewGibbs/HNABEMLAB,
  (2019).

\bibitem{git:REEF}
\leavevmode\vrule height 2pt depth -1.6pt width 23pt, {\em Reef: Residue
  enhanced embedding formulae}, https://github.com/AndrewGibbs/REEF,  (2021).

\bibitem{GiChLaMo:21}
{\sc A.~Gibbs, S.~N. Chandler-Wilde, S.~Langdon, and A.~Moiola}, {\em A
  high-frequency boundary element method for scattering by a class of multiple
  obstacles}, IMA J. Numer. Anal., 41 (2021), pp.~1197--1239.

\bibitem{GiHeHuPa:20}
{\sc A.~Gibbs, D.~P. Hewett, D.~Huybrechs, and E.~Parolin}, {\em Fast hybrid
  numerical-asymptotic boundary element methods for high frequency screen and
  aperture problems based on least-squares collocation}, SN Partial
  Differential Equations and Applications, 1 (2020), p.~21.

\bibitem{GoLo:13}
{\sc G.~H. Golub and C.~F. Van~Loan}, {\em Matrix computations}, Johns Hopkins
  University Press, Baltimore, MD, fourth~ed., 2013.

\bibitem{hargreaves2015high}
{\sc J.~Hargreaves, Y.~W. Lam, S.~Langdon, and D.~P. Hewett}, {\em A
  high-frequency bem for 3d acoustic scattering}, in The 22nd International
  Conference on Sound and Vibration, 2015.

\bibitem{HeLaCh:15}
{\sc D.~P. Hewett, S.~Langdon, and S.~N. Chandler-Wilde}, {\em A
  frequency-independent boundary element method for scattering by
  two-dimensional screens and apertures}, IMA J. Numer. Anal., 35 (2015),
  pp.~1698--1728.

\bibitem{HeLaMe:13}
{\sc D.~P. Hewett, S.~Langdon, and J.~M. Melenk}, {\em A high frequency {$hp$}
  boundary element method for scattering by convex polygons}, SIAM J. Numer.
  Anal., 51 (2013), pp.~629--653.

\bibitem{Hi04}
{\sc N.~J. Higham}, {\em The numerical stability of barycentric {L}agrange
  interpolation}, IMA J. Numer. Anal., 24 (2004), pp.~547--556.

\bibitem{HoPa:92}
{\sc Y.~P. Hong and C.-T. Pan}, {\em Rank-revealing {$QR$} factorizations and
  the singular value decomposition}, Math. Comp., 58 (1992), pp.~213--232.

\bibitem{IoPaPe:91}
{\sc N.~I. Ioakimidis, K.~E. Papadakis, and E.~A. Perdios}, {\em Numerical
  evaluation of analytic functions by {C}auchy's theorem}, BIT, 31 (1991),
  pp.~276--285.

\bibitem{Ma:06}
{\sc P.~A. Martin}, {\em Multiple scattering: Interaction of time-harmonic
  waves with $N$ obstacles}, vol.~107 of Encyclopedia of Mathematics and its
  Applications, Cambridge University Press, Cambridge, 2006.

\bibitem{MiHoTr:00}
{\sc M.~I. Mishchenko, J.~W. Hovenier, and L.~D. Travis}, {\em Light scattering
  by nonspherical particles: Theory, measurements, and applications},
  Measurement Science and Technology, 11 (2000), p.~1827.

\bibitem{Mo:95}
{\sc P.~Monk}, {\em The near field to far field transformation}, COMPEL, 14
  (1995), pp.~41--56.

\bibitem{SaVi:35}
{\sc B.~Sadiq and D.~Viswanath}, {\em Barycentric {H}ermite interpolation},
  SIAM J. Sci. Comput., 35 (2013), pp.~A1254--A1270.

\bibitem{SaSc:11}
{\sc S.~A. Sauter and C.~Schwab}, {\em Boundary element methods}, vol.~39 of
  Springer Series in Computational Mathematics, Springer-Verlag, Berlin, 2011.

\bibitem{SkCrShVa:10}
{\sc E.~A. Skelton, R.~V. Craster, A.~V. Shanin, and V.~Valyaev}, {\em
  Embedding formulae for scattering by three-dimensional structures}, Wave
  Motion, 47 (2010), pp.~299--317.

\bibitem{trefethen2020quantifying}
{\sc L.~N. Trefethen}, {\em Quantifying the ill-conditioning of analytic
  continuation}, BIT Numerical Mathematics, 60 (2020), pp.~901--915.

\bibitem{TreWe:14}
{\sc L.~N. Trefethen and J.~A.~C. Weideman}, {\em The exponentially convergent
  trapezoidal rule}, SIAM Rev., 56 (2014), pp.~385--458.

\end{thebibliography}
	\end{document}